\documentclass[final]{siamltex}

\renewcommand{\div}{\operatorname{div}}
\usepackage{amssymb}
\usepackage[leqno]{amsmath}
\usepackage{mathrsfs}
\usepackage{stmaryrd}
\usepackage{chemarrow}
\usepackage{algorithm2e}
\usepackage{enumerate}
\usepackage[pdftex,bookmarksnumbered,bookmarksopen,colorlinks,linkcolor=red,anchorcolor=black,citecolor=blue,urlcolor=blue]{hyperref}
\usepackage[all]{xy}

\newtheorem{remark}[theorem]{Remark}

\title{Multigrid Methods for Hellan-Herrmann-Johnson Mixed Method of Kirchhoff Plate Bending Problems}%
\author{
Long Chen\thanks{Department of Mathematics, University of California at Irvine, Irvine, CA 92697, USA, and Beijing Institute for Scientific and Engineering Computing, Beijing University of Technology, Beijing 100124, China (chenlong@math.uci.edu).
The work of this author was supported by the National Science Foundation (NSF) DMS-1418934, and
in part by the Sea Poly Project of Beijing Overseas Talents and the National Natural Science Foundation of China Project 11671159. This work was finished when the first
author visited Peking University in the fall of 2015. He would like to thank Peking University for
the support and hospitality, as well as for their exciting research atmosphere.
}
\and
Jun Hu\thanks{LMAM and School of Mathematical Sciences, Peking University, Beijing 100871, China (hujun@math.pku.edu.cn). The work of this author was supported by the National Natural Science Foundation of China Projects 11625101,  91430213 and 11421101.}
\and
Xuehai Huang\thanks{Corresponding author. College of Mathematics and Information Science, Wenzhou University, Wenzhou 325035, China (xuehaihuang@gmail.com).
The work of this author was supported by the National Natural Science Foundation of China Projects 11771338 and 11671304, Zhejiang Provincial
Natural Science Foundation of China Projects LY17A010010, LY15A010015 and LY15A010016, and Wenzhou Science and Technology Plan Project G20160019.}
}%

\begin{document}

\maketitle

\begin{abstract}
A V-cycle multigrid method for the Hellan-Herrmann-Johnson (HHJ) discretization of the Kirchhoff plate bending problems is developed in this paper.
It is shown that the contraction number of the V-cycle multigrid HHJ mixed method is bounded away from one uniformly with respect to the mesh size. The uniform convergence is achieved for the V-cycle multigrid method with only one smoothing step and without full elliptic regularity. The key is a stable decomposition of the kernel space which is derived from an exact sequence of the HHJ mixed method, and the strengthened Cauchy Schwarz inequality. Some numerical experiments are provided to confirm the proposed V-cycle multigrid method. The exact sequences of the HHJ mixed method and the corresponding commutative diagram is of some interest independent of the current context.
\end{abstract}

\begin{keywords}
Kirchhoff plates, Hellan-Herrmann-Johnson mixed method, multigrid method, exact sequence, stable decomposition
\end{keywords}


\section{Introduction}
We consider multigrid methods for solving the saddle point system arising from the Hellan-Herrmann-Johnson (HHJ) mixed method discretization (cf.~\cite{Hellan1967, Herrmann1967, Johnson1973}) of a fourth order equation: the Kirchhoff plate bending problem.

Linear systems arising from discretization of fourth order partial differential equations are difficult to solve due to the poor spectral properties.
%
For $C^1$ conforming finite element methods of the biharmonic equation, some multigrid methods are studied in~\cite{Zhang1989, BrambleZhang1995,XuLi1998,Zhao2005}. In practice since it is hard to construct $C^1$ finite elements, nonconforming finite element methods (cf.~\cite{Ciarlet1978, LascauxLesaint1975,WangShiXu2007a}), notably the Morley element (cf.~\cite{Morley1968, WangXu2006, WangShiXu2007a, WangXu2013}), Zienkiewicz element (cf.~\cite{BazeleyCheungIronsZienkiewicz1965, WangShiXu2007}) and Adini element (cf.~\cite{AdiniClough1961, LascauxLesaint1975,WangShiXu2007a}), are favored for the fourth order equation.
Optimal-order nonconforming multigrid methods with the full regularity assumption are developed in~\cite{Brenner1989, PeiskerRustStein1990, ZhouFeng1993,XuLi1999,ShiXu2000}.
Without assuming full elliptic regularity, similar results are obtained in~\cite{Wang1994a, Brenner1999a,Zhao2004}.
For $C^0$ interior penalty methods of fourth order equations in~\cite{BrennerSung2005, EngelGarikipatiHughesLarsonEtAl2002}, it is proved in~\cite{BrennerSung2006} that V-cycle, F-cycle and W-cycle multigrid algorithms are uniform contractions. Standard mutligrid solvers for the Poisson operator are used to design efficient smoothers.
An algebraic multigrid method by smooth aggregation is developed for the fourth order elliptic problems in~\cite{VanvekMandelBrezina1996}.
In all of these work, special intergrid transfer operators are necessary for both these conforming and nonconforming multigrid methods, since either the underlying finite element spaces are non-nested or the quadratic forms are non-inherited. The contraction number of V-cycle, W-cycle or F-cycle multigrid method can be proved to be less than one uniformly with respect to the mesh level provided that the number of smoothing steps is large enough.

We shall develop a multigrid method for the Hellan-Herrmann-Johnson discretization of the Kirchhoff plate bending problems in the mixed form. The resulting linear system is in the following saddle point form
\begin{equation}\label{intro:saddle}
\begin{pmatrix}
M & B^T \\
B & O
\end{pmatrix}
\begin{pmatrix}
\boldsymbol \sigma\\
u
\end{pmatrix}
=
\begin{pmatrix}
0 \\
-f
\end{pmatrix},
\end{equation}
which is considered harder to solve than the symmetric positive counterpart due to the indefiniteness of the saddle point system. To this end, the hybridization technique is applied to the HHJ mixed method by introducing a Lagrange multiplier, which changes the saddle point system to a symmetric positive definite (SPD) problem (cf.~\cite{Veubeke1965, ArnoldBrezzi1985}). It is shown in~\cite{ArnoldBrezzi1985} that the resulted SPD problem in the lowest order HHJ mixed method is equivalent to a modified Morley method. As we mentioned earlier, however, multigrid algorithms for the Morley element method have been only proved to be optimal with special intergrid transfer operators and large enough number of smoothing steps.

We shall apply the approach developed in~\cite{ChenSaddle} to design an effective multigrid methods for solving the equivalent linear system of \eqref{intro:saddle}, whose mixed finite element method is to find $(\widetilde{\boldsymbol{\sigma}}_h, u_h)\in \mathcal{V}_h\times \mathcal{P}_h$ such that
\begin{align}
&a(\widetilde{\boldsymbol{\sigma}}_h,\boldsymbol{\tau})+b(\boldsymbol{\tau}, u_h)=-a(\boldsymbol{\Pi}_h\boldsymbol{\sigma}_0,\boldsymbol{\tau}) \quad \forall\,\boldsymbol{\tau}\in\mathcal{V}_h,  \label{mfemnew1intro}\\
&b(\widetilde{\boldsymbol{\sigma}}_h, v)\quad\quad\quad\quad\;\;\;=0 \quad\quad\quad\quad\quad\quad\; \forall\,v\in \mathcal{P}_h. \label{mfemnew2intro}
\end{align}
The smoother of our multigrid method is a multiplicative Schwarz smoother based on a multilevel decomposition of the null space $\ker(B)$. Since the finite element spaces of the HHJ mixed method are nested, the coarse-to-fine intergrid transfer operator are simply the natural injection.

The key to the analysis and the algorithm is a stable multilevel decomposition of the null space $\ker(\mathrm{div}\boldsymbol{\mathrm{div}})$. To this end, we first establish  exact sequences for the HHJ mixed method of Kirchhoff plates in both continuous and discrete levels.
By the discrete exact sequence,
the mixed method \eqref{mfemnew1intro}-\eqref{mfemnew2intro} is equivalent to find $\boldsymbol{\phi}_h\in \widetilde{\mathcal{S}}_h$ such that \cite{KrendlRafetsederZulehner2016}
\begin{equation*}
a(\nabla ^s\times\boldsymbol{\phi}_h, \nabla ^s\times\boldsymbol{\psi})=-a(\boldsymbol{\Pi}_h\boldsymbol{\sigma}_0, \nabla ^s\times\boldsymbol{\psi}) \quad \forall\,\boldsymbol{\psi}\in\widetilde{\mathcal{S}}_h
\end{equation*}
with $\widetilde{\boldsymbol{\sigma}}_h=\nabla ^s\times\boldsymbol{\phi}_h$.
After achieving a decomposition of the finite element space for the stress based on the discrete exact sequence for the HHJ mixed method,
a stable decomposition and the strengthened Cauchy Schwarz inequality are derived using the standard technique as in~\cite{Xu1992}.
%
Then according to the theoretical results developed in~\cite{ChenSaddle}, the contraction number of our V-cycle multigrid HHJ mixed method is bounded away from one uniformly with respect to the mesh size with even only one smoothing step. Since a stable decomposition is obtained using the $L^2$-projection, the full regularity assumption is not needed neither in our approach. As far as we know, our V-cycle multigrid method is the first work possessing these two merits among the multigrid methods for solving the fourth order partial differential equation directly.

Although the multigrid method used here and its convergence follow from the framework developed in~\cite{ChenSaddle}, this example has special feature which lead to rather difficult analysis than examples considered in~\cite{ChenSaddle}. Furthermore, the Hilbert complex for the HHJ mixed method revealed in this paper is of some interest independent of the current context and will play a central role in the design and analysis of the HHJ mixed method~\cite{ArnoldFalkWinther2006}, c.f. the convergence of adaptive finite element methods for the HHJ mixed method established in \cite{HuangHuangXu2011}. We emphasize such a contribution by listing the commutative diagram for the HHJ mixed method as follows. Details on the spaces and interpolation operators can be found in Section 2.2.
\[
\begin{array}{c}
\xymatrix{
  \overline{\boldsymbol{P}}_1(\Omega; \mathbb{R}^2) \ar[r]^{\subset} & \boldsymbol{H}^{1}(\Omega; \mathbb{R}^2) \ar[d]^{\boldsymbol{I}_h} \ar[r]^-{\nabla ^s\times}
                & \boldsymbol{H}^{-1}(\mathrm{div}\boldsymbol{\mathrm{div}},\Omega; \mathbb{S}) \ar[d]^{\boldsymbol{\Pi}_h}   \ar[r]^-{\mathrm{div}\boldsymbol{\mathrm{div}}} & \ar[d]^{Q_h}H^{-1}(\Omega) \ar[r]^{} & 0 \\
 \overline{\boldsymbol{P}}_1(\Omega; \mathbb{R}^2) \ar[r]^{\subset} & \mathcal{S}_h \ar[r]^{\nabla ^s\times}
                & \mathcal{V}_h   \ar[r]^{(\mathrm{div}\boldsymbol{\mathrm{div}})_h} &  \mathcal{P}_h \ar[r]^{}& 0    }
\end{array}.
\]

%
%

The rest of this paper is organized as follows. The HHJ mixed method for Kirchhoff plates and the corresponding exact sequence and commutative diagram are presented in Section 2.  Then we construct a stable decomposition and prove the strengthened Cauchy Schwarz inequality for the HHJ mixed method in
Section 3. In Section 4, we show and analyze the V-cycle multigrid method for the HHJ mixed method.
Some numerical experiments are given to testify our multigrid method in Section 4 as well.
\section{Mixed Methods for the Plate Bending Problem}
Assume a thin plate occupies a bounded simply connected polygonal domain $\Omega\subset\mathbb{R}^{2}$.
Then the mathematical model describing the deflection $u$ of the plate is governed by (cf.~\cite{FengShi1996,Reddy2006})
\begin{equation}
\left\{
\begin{aligned}
&\mathcal{C}\boldsymbol{\sigma}=-\boldsymbol{\nabla}^2u\;\;\text{in }\Omega,\\
&\div\boldsymbol \div\boldsymbol{\sigma}=-f\;\;\text{in
}\Omega, \\
&u=\partial_{\boldsymbol{n}}u=0\;\;\text{on }\partial\Omega,
\end{aligned}
\right.  \label{kirchhoffplate}%
\end{equation}
where $\boldsymbol{n}$ is the unit
outward normal to $\partial\Omega$, $\boldsymbol{\nabla}$ is the usual gradient
operator, $\div\boldsymbol \div$ stands for the divergence operator acting
on vector-valued (tensor-valued)  functions (cf.~\cite{Reddy2006}).
Here, $\mathcal{C}$ is a symmetric and positive definite operator defined as follows: for any second-order tensor $\boldsymbol{\tau}$,
\[
\mathcal{C}\boldsymbol{\tau}:=\frac{1}{1-\nu}\boldsymbol{\tau}-\frac{\nu}{1-\nu^2}({\rm
tr}\boldsymbol{\tau})
   \mathcal{I}
\]
with $\mathcal{I}$ a second order identity tensor, ${\rm tr}$ the trace operator
acting on second order tensors, and $\nu\in L^{\infty}(\Omega)$ the Poisson ratio satisfying $\inf\limits_{x\in\Omega}\nu\geq 0$ and $\sup\limits_{x\in\Omega}\nu<0.5$.

\subsection{Hellan-Herrmann-Johnson Method}
Denote the space of all symmetric $2\times2$ tensor by $\mathbb{S}$.  Given a bounded domain $G\subset\mathbb{R}^{2}$ and a
non-negative integer $m$, let $H^m(G)$ be the usual Sobolev space of functions
on $G$, and $\boldsymbol{H}^m(G; \mathbb{X})$ be the usual Sobolev space of functions taking values in the finite-dimensional vector space $\mathbb{X}$ for $\mathbb{X}$ being $\mathbb{S}$ or $\mathbb{R}^2$. The corresponding norm and semi-norm are denoted respectively by
$\Vert\cdot\Vert_{m,G}$ and $|\cdot|_{m,G}$.  If $G$ is $\Omega$, we abbreviate
them by $\Vert\cdot\Vert_{m}$ and $|\cdot|_{m}$,
respectively. Let $H_{0}^{m}(G)$ be the closure of $C_{0}^{\infty}(G)$ with
respect to the norm $\Vert\cdot\Vert_{m,G}$.  $P_m(G)$ stands for the set of all
polynomials in $G$ with the total degree no more than $m$, and $\boldsymbol{P}_m(G; \mathbb{X})$ denotes the tensor or vector version of $P_m(G)$ for $\mathbb{X}$ being $\mathbb{S}$ or $\mathbb{R}^2$, respectively.

Let $\{\mathcal {T}_h\}_{h>0}$ be a regular family of triangulations
of $\Omega$.  For each $K\in\mathcal{T}_h$, denote by $\boldsymbol{n}_K = (n_1, n_2)^T$ the
unit outward normal to $\partial K$ and write $\boldsymbol{t}_K:= (t_1, t_2)^T = (-n_2, n_1)^T$, a unit vector
tangent to $\partial K$. Without causing any confusion, we will abbreviate $\boldsymbol{n}_K$ and $\boldsymbol{t}_K$ as $\boldsymbol{n}$ and $\boldsymbol{t}$ respectively for simplicity.
Let $\mathcal{E}_h$ be the union of all edges
of the triangulation $\mathcal {T}_h$ and $\mathcal{E}^i_h$ the union of all
interior edges of the triangulation $\mathcal {T}_h$.
Set for each $K\in\mathcal{T}_h$
\[
\mathcal{E}_h(K):=\{e\in\mathcal{E}_h: e\subset\partial K\}, \quad \mathcal{E}_h^i(K):=\{e\in\mathcal{E}_h^i: e\subset\partial K\}.
\]
For any $e\in\mathcal{E}_h$,
fix a unit normal vector $\boldsymbol{n}_e:= (n_1, n_2)^T$ and a unit tangent vector $\boldsymbol{t}_e := (-n_2, n_1)^T$. For a column vector function $\boldsymbol \phi = (\phi_1, \phi_2)^T$, differential operators for scalar functions will be applied row-wise to produce a matrix function. Similarly for a matrix function, differential operators for vector functions are applied row-wise.
Discrete differential operator $\boldsymbol \div_h$ is defined as
the elementwise counterpart of $\boldsymbol \div$ with respect to the triangulation $\mathcal{T}_h$.
For a second order tensor-valued function $\boldsymbol{\tau}$, set
\[
M_{n}(\boldsymbol{\tau}):=\boldsymbol{n}^T\boldsymbol{\tau}\boldsymbol{n}, \quad M_{nt}(\boldsymbol{\tau}):=\boldsymbol{t}^T\boldsymbol{\tau}\boldsymbol{n},
\]
on each edge $e\in\mathcal{E}_h$.
Next, we introduce jumps on edges. Consider two adjacent triangles $K^+$ and $K^-$ sharing an interior edge $e$.
Denote by $\boldsymbol{n}^+$ and $\boldsymbol{n}^-$ the unit outward normals
to the common edge $e$ of the triangles $K^+$ and $K^-$, respectively.
For a scalar-valued function $v$, write $v^+:=v|_{K^+}$ and $v^-
:=v|_{K^-}$.   Then define jumps on $e$ as
follows:
\[
[v]:=v^+\boldsymbol{n}_e\cdot\boldsymbol{n}^++v^-\boldsymbol{n}_e\cdot\boldsymbol{n}^-.
\]
On an edge $e\subset K$ lying on the boundary $\partial\Omega$, the above term is
defined by
\[
[v]:=v\boldsymbol{n}_e\cdot\boldsymbol{n}_K.
\]
For any second order tensor-valued functions $\boldsymbol{\sigma}$ and $\boldsymbol{\tau}$, set
\[
\boldsymbol{\sigma}:\boldsymbol{\tau}:=\sum_{i,j=1}^2\boldsymbol{\sigma}_{ij}\boldsymbol{\tau}_{ij}.
\]
Throughout this paper, we use ``$\lesssim\cdots $" to mean that ``$\leq C\cdots$", where $C$ is a generic positive constant independent of the mesh size $h$, which may take different values at different appearances.

Then we define some Hilbert spaces.
Based on the triangulation $\mathcal {T}_h$, let
\begin{align*}
&\mathcal{V}:=\left\{\boldsymbol{\tau}\in\boldsymbol{L}^2(\Omega; \mathbb{S}):
\boldsymbol{\tau}|_K\in \boldsymbol{H}^1(K; \mathbb{S}) \quad \forall\,K\in\mathcal{T}_h \textrm{ and } [M_n(\boldsymbol{\tau})]|_e=0\;\, \forall\,e\in\mathcal
{E}_h^i \right\},\\
&\mathcal{P}:=\left\{v\in H_0^1(\Omega): v|_K\in H^2(K)\quad \forall\,K\in\mathcal
{T}_h\right\}.
\end{align*}
The corresponding finite element spaces are given by
\begin{align*}
&\mathcal{S}_h:=\left\{\boldsymbol{\phi}\in\boldsymbol{H}^1(\Omega; \mathbb{R}^2):
\boldsymbol{\phi}|_K\in \boldsymbol{P}_{r}(K; \mathbb{R}^2) \quad \forall\,K\in\mathcal{T}_h \right\},\\
&\mathcal{V}_h:=\left\{\boldsymbol{\tau}\in\mathcal{V}:
\boldsymbol{\tau}|_K\in \boldsymbol{P}_{r-1}(K; \mathbb{S}) \quad \forall\,K\in\mathcal{T}_h \right\},\\
&\mathcal{P}_h:=\left\{v\in H_0^1(\Omega): v|_K\in P_r(K)\quad \forall\,K\in\mathcal
{T}_h\right\}
\end{align*}
with integer $r\geq1$.

With previous preparation, the Hellan-Herrmann-Johnson (HHJ) mixed method (cf.~\cite{Hellan1967, Herrmann1967, Johnson1973}) for problem~\eqref{kirchhoffplate} is given as follows: Find $(\boldsymbol{\sigma}_h, u_h)\in \mathcal{V}_h\times \mathcal{P}_h$ such that
\begin{align}
&a(\boldsymbol{\sigma}_h,\boldsymbol{\tau})+b(\boldsymbol{\tau}, u_h)=0 \quad\quad\quad\quad\quad\quad \forall\,\boldsymbol{\tau}\in\mathcal{V}_h, \label{mfem1} \\
&b(\boldsymbol{\sigma}_h, v)\quad\quad\quad\quad\quad=-\int_{\Omega}fv\, {\rm d}x \quad\;\;\; \forall\,v\in \mathcal{P}_h, \label{mfem2}
\end{align}
where
\begin{align*}
a(\boldsymbol{\sigma},\boldsymbol{\tau})&:=\int_{\Omega}\mathcal{C}\boldsymbol{\sigma}:\boldsymbol{\tau}\,{\rm d}x \quad \forall\,\boldsymbol{\sigma},\boldsymbol{\tau}\in \mathcal{V},\\
b(\boldsymbol{\tau}, v)&:= -\int_{\Omega}(\boldsymbol \div_h\boldsymbol{\tau})\cdot\boldsymbol{\nabla}v\, {\rm d}x + \sum_{K\in\mathcal{T}_h}\int_{\partial K}M_{nt}(\boldsymbol{\tau})\partial_{\boldsymbol{t}}v\, {\rm d}s\quad \forall\,\boldsymbol{\tau}\in \mathcal{V}, v\in \mathcal{P}.
\end{align*}

The boundary condition for the deflection $u=0$ on $\partial \Omega$ is imposed into the space $\mathcal P_h$ whereas the boundary condition for the rotation $\partial_{\boldsymbol n} u=0$ on $\partial \Omega$ is imposed weakly in the variational form \eqref{mfem1}.
If the plate is simply supported along the boundary, i.e. the boundary condition is now $u=0, M_n(\boldsymbol{\sigma})=0$ on $\partial\Omega$, we only need to modify $\mathcal{V}$ as
\[
\mathcal{V}_0:=\left\{\boldsymbol{\tau}\in\mathcal{V}:
M_n(\boldsymbol{\tau})=0 \textrm{ on } \partial\Omega \right\}.
\]

It was shown in~\cite{BabuvskaOsbornPitkaranta1980, FalkOsborn1980,BoffiBrezziFortin2013} that the HHJ mixed method~\eqref{mfem1}-\eqref{mfem2} is well posed. And the inf-sup condition holds as follows (cf. \cite[Lemma 4.2]{HuangHuangXu2011})
\[
\|v_h\|_{2,h}\lesssim\sup_{\boldsymbol{\tau}_h\in\mathcal{V}_h}\frac{b(\boldsymbol{\tau}_h, v_h)}{\|\boldsymbol{\tau}_h\|_{0,h}}\quad \forall~v_h\in\mathcal{P}_h.
\]
where mesh dependent norms are
\begin{align*}
\|v\|_{2,h}^2&:=\sum_{K\in\mathcal{T}_h}\|v\|_{2,K}^2+\sum_{e\in\mathcal{E}_h}h_e^{-1}\|[\partial_{\boldsymbol{n}_e}v]\|_{0,e}^2,\\
\|\boldsymbol{\tau}\|_{0,h}^2&:=\|\boldsymbol{\tau}\|_0^2+\sum_{e\in\mathcal{E}_h}h_e\|M_n(\boldsymbol{\tau})\|_{0,e}^2.
\end{align*}
And it possesses the optimal {\it a priori} error estimates provided that $\boldsymbol \sigma$ and $u$ are smooth enough:
\begin{align*}
\|\boldsymbol{\sigma}-\boldsymbol{\sigma}_h\|_0&\lesssim h^r\|\boldsymbol{\sigma}\|_{r},\\
\|u-u_h\|_1&\lesssim h^r(\|\boldsymbol{\sigma}\|_{r}+\|u\|_{r+1}).
\end{align*}

Reliable and efficient {\it a posteriori} error estimators, as well as the convergence of an adaptive  HHJ mixed method, can be found in~\cite{HuangHuangXu2011}.

\subsection{Hilbert Complex for the HHJ Mixed Method}
In this section, we shall derive the exact sequence and commutative diagram for the HHJ mixed method~\eqref{mfem1}-\eqref{mfem2}.


For a vector-valued function $\boldsymbol{\phi}=(\phi_1, \phi_2)^{T}$,
denote by $\boldsymbol{\phi}^{\perp}:=(-\phi_2, \phi_1)^T$ the vector perpendicular to $\boldsymbol{\phi}$.
The standard symmetric gradient operator is
\[
\boldsymbol{\varepsilon}(\boldsymbol{\phi})=\frac{1}{2}\left(\boldsymbol{\nabla}\boldsymbol{\phi}+(\boldsymbol{\nabla}\boldsymbol{\phi})^T\right).
\]
The symmetric curl operator will be defined analogically by
$$
\nabla ^s\times\boldsymbol \phi :=
\frac{1}{2}\left ( \mathbf{curl} \boldsymbol \phi +  (\mathbf{curl}\boldsymbol \phi)^T\right ).
$$
%
Let
\[
\overline{\boldsymbol{P}}_1(\Omega; \mathbb{R}^2):=\textrm{span}\left\{
\left(\begin{array}{c}
0 \\
1
\end{array}
\right), \left(\begin{array}{c}
1 \\
0
\end{array}
\right),  \left(\begin{array}{c}
x_1 \\
x_2
\end{array}
\right)\right\}.
\]
It is easy to see that $\overline{\boldsymbol{P}}_1^{\textrm{Rot}}(\Omega; \mathbb{R}^2)$ is exactly the rigid body motion space where
\[
\overline{\boldsymbol{P}}_1^{\textrm{Rot}}(\Omega; \mathbb{R}^2):=\{\boldsymbol{\phi}\in\boldsymbol{L}^2(\Omega; \mathbb{R}^2): \boldsymbol{\phi}^{\perp}\in\overline{\boldsymbol{P}}_1(\Omega; \mathbb{R}^2)\}.
\]

\begin{lemma}\label{lem:1}
The following sequence for Kirchhoff plates
\begin{equation}\label{exactsequence1}
\overline{\boldsymbol{P}}_1(\Omega; \mathbb{R}^2)\autorightarrow{$\subset$}{} \boldsymbol{C}^{\infty}(\Omega; \mathbb{R}^2)\autorightarrow{$\nabla ^s\times$}{} \boldsymbol{C}^{\infty}(\Omega; \mathbb{S}) \autorightarrow{$\mathrm{div}\boldsymbol{\mathrm{div}}$}{} C^{\infty}(\Omega)\autorightarrow{}{}0
\end{equation}
is an exact complex.
\end{lemma}
\begin{proof}
By direct computation, it is easy to see that \eqref{exactsequence1} is a complex, i.e. $\nabla ^s\times(\overline{\boldsymbol{P}}_1) = 0$ and $\mathrm{div}\boldsymbol{\mathrm{div}}~\nabla ^s\times = 0$. We then verify the exactness.

Let us first show that $\ker(\nabla ^s\times)=\overline{\boldsymbol{P}}_1(\Omega; \mathbb{R}^2)$. For any $\boldsymbol{\phi}\in\boldsymbol{C}^{\infty}(\Omega; \mathbb{R}^2)$ satisfying $\nabla ^s\times\boldsymbol{\phi}=\boldsymbol{0}$, it holds
\[
\nabla ^s\times\boldsymbol{\phi} = \boldsymbol{L}^T\boldsymbol{\varepsilon}(\boldsymbol \phi^{\bot})\boldsymbol{L}=\boldsymbol{0}.
\]
where $\boldsymbol{L}=\left(
\begin{array}{cc}
0 & -1 \\
1 & 0
\end{array}
\right)$.
Thus we have
\[
\boldsymbol{\varepsilon}(\boldsymbol \phi^{\bot})=\boldsymbol{0},
\]
which implies $\boldsymbol{\phi}\in\overline{\boldsymbol{P}}_1(\Omega; \mathbb{R}^2)$.

Next we demonstrate that $\ker(\mathrm{div}\boldsymbol{\mathrm{div}})=\nabla ^s\times\boldsymbol{C}^{\infty}(\Omega; \mathbb{R}^2)$ using the similar argument adopted in~\cite[Lemma~1]{Beirao-da-VeigaNiiranenStenberg2007} and \cite[Lemma~3.1]{HuangHuangXu2011}.
First of all, $\nabla ^s\times\boldsymbol{C}^{\infty}(\Omega; \mathbb{R}^2) \subset \ker(\mathrm{div}\boldsymbol{\mathrm{div}})$ by direct computation.
For any $\boldsymbol{\tau}\in \ker(\mathrm{div}\boldsymbol{\mathrm{div}})$, there exists $v\in C^{\infty}(\Omega)$ such that
$\boldsymbol{\mathrm{div}}\boldsymbol{\tau}=\mathbf{curl}v=-\boldsymbol{\mathrm{div}}(v\boldsymbol{L}),$
which implies
$\boldsymbol{\mathrm{div}}(\boldsymbol{\tau}+v\boldsymbol{L})=0.$
Hence
there exists a vector function $\boldsymbol{\phi}\in\boldsymbol{C}^{\infty}(\Omega; \mathbb{R}^2)$ satisfying
\[
\boldsymbol{\tau}+v\boldsymbol{L}=\mathbf{curl}\boldsymbol{\phi}.
\]
Since $\boldsymbol{\tau}$ is symmetric, we have $\boldsymbol{\tau}=\nabla ^s\times\boldsymbol{\phi}$.
Thus $\ker(\mathrm{div}\boldsymbol{\mathrm{div}})\subset \nabla ^s\times\boldsymbol{C}^{\infty}(\Omega; \mathbb{R}^2)$.

Finally we show that $\mathrm{div}\boldsymbol{\mathrm{div}}\boldsymbol{C}^{\infty}(\Omega; \mathbb{S})=C^{\infty}(\Omega)$. By the elasticity complex in \cite[p. 405]{ArnoldWinther2002}, the divergence operator $\boldsymbol \div :\boldsymbol{C}^{\infty}(\Omega; \mathbb{S})\to \boldsymbol{C}^{\infty}(\Omega; \mathbb{R}^2)$ is surjective. And due to the de Rham complex in \cite[p. 27]{ArnoldFalkWinther2006}, the divergence operator $\div :\boldsymbol{C}^{\infty}(\Omega; \mathbb{R}^2)\to C^{\infty}(\Omega)$ is also surjective. Hence we have $\mathrm{div}\boldsymbol{\mathrm{div}}\boldsymbol{C}^{\infty}(\Omega; \mathbb{S})=C^{\infty}(\Omega)$.
\end{proof}

We then derive an exact sequence with less smoothness. To this end, we define $B: \mathcal V\to \mathcal P^{\prime}$ as
$$
\langle B\boldsymbol{\tau}, v\rangle :=b(\boldsymbol{\tau}, v) \quad \forall~ v\in \mathcal P.
$$
For any $(\boldsymbol{\tau}, v)\in\mathcal V\times \mathcal P$ with $v\in H_0^2(\Omega)$, it follows from an integration by parts and the fact $\left [M_n(\boldsymbol{\tau}) \right ]|_{\mathcal{E}_h^i}=0$ that
\begin{align}
\langle B\boldsymbol{\tau}, v\rangle =& \int_{\Omega}\boldsymbol{\tau}:\boldsymbol{\nabla}^2v\, {\rm d}x - \sum_{K\in\mathcal{T}_h}\int_{\partial K}(\boldsymbol{\tau}\boldsymbol{n})\cdot\boldsymbol{\nabla}v\, {\rm d}s+\sum_{K\in\mathcal{T}_h}\int_{\partial K}M_{nt}(\boldsymbol{\tau})\partial_{\boldsymbol{t}}v\, {\rm d}s \notag\\
=&\int_{\Omega}\boldsymbol{\tau}:\boldsymbol{\nabla}^2v\, {\rm d}x - \sum_{K\in\mathcal{T}_h}\int_{\partial K}M_{n}(\boldsymbol{\tau})\partial_{\boldsymbol{n}}v\, {\rm d}s \notag\\
=&\int_{\Omega}\boldsymbol{\tau}:\boldsymbol{\nabla}^2v\, {\rm d}x=\langle \div\boldsymbol \div\boldsymbol{\tau}, v\rangle_{H^{-2}(\Omega)\times H_0^2(\Omega)}. \label{eq:temp8}
\end{align}
On the other side,
for any $(\boldsymbol{\tau}, v)\in\mathcal V\times \mathcal P$ with $\boldsymbol{\tau}\in \boldsymbol H(\boldsymbol \div,\Omega; \mathbb S): = \{\boldsymbol \tau \in L^2(\Omega; \mathbb S): \boldsymbol\div \boldsymbol \tau \in \boldsymbol{L}^{2}(\Omega; \mathbb{R}^2)\}$, since $v\in\mathcal P$ implies $v$ being continuous in $\Omega$, it follows from the fact $\left [M_{nt}(\boldsymbol{\tau}) \right ]|_{\mathcal{E}_h^i}=0$ that
\begin{align*}
\langle B\boldsymbol{\tau}, v\rangle =&-\int_{\Omega}(\boldsymbol \div\boldsymbol{\tau})\cdot\boldsymbol{\nabla}v\, {\rm d}x + \sum_{K\in\mathcal{T}_h}\int_{\partial K}M_{nt}(\boldsymbol{\tau})\partial_{\boldsymbol{t}}v\, {\rm d}s \\
=&-\int_{\Omega}(\boldsymbol \div\boldsymbol{\tau})\cdot\boldsymbol{\nabla}v\, {\rm d}x=\langle \div\boldsymbol \div\boldsymbol{\tau}, v\rangle_{H^{-1}(\Omega)\times H_0^1(\Omega)}.
\end{align*}
Therefore the bilinear form $b(\cdot,\cdot)$ can be defined either on $\boldsymbol H(\boldsymbol \div,\Omega; \mathbb S)\times H_0^1(\Omega)$ as $B=\div\boldsymbol \div$ in $H^{-1}(\Omega)$ distribution sense or $\boldsymbol L^2(\Omega)\times H^{2}_0(\Omega)$ with $B=\div\boldsymbol \div$ in $H^{-2}(\Omega)$ distribution sense. However, conforming finite element spaces of $\boldsymbol H(\boldsymbol \div,\Omega; \mathbb S)$ or $H^{2}_0(\Omega)$ are difficult to construct. Until this century, $\boldsymbol H(\boldsymbol \div,\Omega; \mathbb S)$ conforming mixed
finite elements with polynomial shape functions were constructed in \cite{Hu2015a,HuZhang2015c, HuZhang2015, ChenHuHuang2017, HuZhang2016, ArnoldWinther2002, AdamsCockburn2005, ArnoldAwanouWinther2008}, and an efficient fast solver on general shape-regular unstructured meshes was recently developed in \cite{ChenHuHuang2017a}.
We strike a balance of the smoothness of these two spaces and understand the bilinear form $b(\cdot,\cdot)$ being defined on $\mathcal V\times \mathcal P$ and thus
$$
\mathrm{div}\boldsymbol{\mathrm{div}}: \boldsymbol{H}^{-1}(\mathrm{div}\boldsymbol{\mathrm{div}},\Omega; \mathbb{S}) \to  H^{-1}(\Omega)
$$
with space $\boldsymbol{H}^{-1}(\div\boldsymbol \div,\Omega; \mathbb S): = \{\boldsymbol \tau \in L^2(\Omega; \mathbb S): \div \boldsymbol\div \boldsymbol \tau \in H^{-1}(\Omega)\}$ which was firstly introduced in \cite{PechsteinSchoberl2011}.

Making use of the similar argument as in Lemma~\ref{lem:1}, we can acquire a Hilbert sequence for Kirchhoff plates as follows.
\begin{lemma}
The following Hilbert sequence for Kirchhoff plates
\begin{equation}\label{eq:hhjexactsequencecontinuous}
\overline{\boldsymbol{P}}_1(\Omega; \mathbb{R}^2)\autorightarrow{$\subset$}{} \boldsymbol{H}^1(\Omega; \mathbb{R}^2)\autorightarrow{$\nabla ^s\times$}{} \boldsymbol{H}^{-1}(\mathrm{div}\boldsymbol{\mathrm{div}},\Omega; \mathbb{S}) \autorightarrow{$\mathrm{div}\boldsymbol{\mathrm{div}}$}{} H^{-1}(\Omega)\autorightarrow{}{}0
\end{equation}
is an exact complex.
\end{lemma}

\begin{remark}\rm
A less smooth exact Hilbert sequence for Kirchhoff plates is
\begin{equation}\label{eq:lesssmoothexactsequencecontinuous}
\overline{\boldsymbol{P}}_1(\Omega; \mathbb{R}^2)\autorightarrow{$\subset$}{} \boldsymbol{L}^2(\Omega; \mathbb{R}^2)\autorightarrow{$\nabla ^s\times$}{} \boldsymbol{H}^{-2}(\mathrm{div}\boldsymbol{\mathrm{div}},\Omega; \mathbb{S}) \autorightarrow{$\mathrm{div}\boldsymbol{\mathrm{div}}$}{} H^{-2}(\Omega)\autorightarrow{}{}0,
\end{equation}
where $\boldsymbol{H}^{-2}(\div\boldsymbol \div,\Omega; \mathbb S): = \{\boldsymbol \tau \in \boldsymbol H^{-1}(\Omega; \mathbb S): \div \boldsymbol\div \boldsymbol \tau \in H^{-2}(\Omega)\}$.
Finite element spaces of $\boldsymbol{H}^{-2}(\div\boldsymbol \div,\Omega; \mathbb S)$ is, however, difficult to construct. Indeed in the HHJ mixed method, the space $\mathcal V$ and $\mathcal V_h$ are not subspaces of $\boldsymbol{H}^{-1}(\div\boldsymbol \div,\Omega; \mathbb S)$ neither. That is, the HHJ mixed method is still a non-conforming method. $\Box$
\end{remark}

\begin{remark}\rm
The dual complex of \eqref{eq:lesssmoothexactsequencecontinuous} is
\[
0\autorightarrow{$\subset$}{} H_0^2(\Omega)\autorightarrow{$\boldsymbol\nabla^2$}{} \boldsymbol{H}_0(\boldsymbol{\mathrm{rot}},\Omega; \mathbb{S}) \autorightarrow{$\boldsymbol{\mathrm{rot}}$}{} \boldsymbol{L}_0^2(\Omega; \mathbb{R}^2)\autorightarrow{}{}0,
\]
where
\[
\boldsymbol H_0(\mathbf{rot},\Omega; \mathbb S): = \{\boldsymbol \tau \in L^2(\Omega; \mathbb S): \mathbf{rot} \boldsymbol \tau \in \boldsymbol{L}^{2}(\Omega; \mathbb{R}^2),\, \textrm{ and }\,\boldsymbol \tau\boldsymbol t=\boldsymbol 0 \, \textrm{ on }\,\partial \Omega\},
\]
\[
\boldsymbol{L}_0^2(\Omega; \mathbb{R}^2): = \{\boldsymbol \phi \in \boldsymbol{L}^2(\Omega; \mathbb{R}^2): \int_{\Omega} \boldsymbol \phi\,\, {\rm d}x=\boldsymbol 0\}.
\]
It is interesting to notice that the last exact sequence is an rotation of the elasticity complex in two dimensions~\cite[(2.1)]{ArnoldWinther2002}. $\Box$
\end{remark}

\medskip

In the discrete level, we shall derive a similar exact sequence for the finite element spaces introduced before. To this end, we first discuss the discretization of the two differential operators $\nabla ^s\times$ and $\mathrm{div}\boldsymbol{\mathrm{div}}$. Since $\nabla ^s\times$ only requires the $H^1$ smoothness, it can be naturally discretized by choosing the finite element space $\mathcal S_h \subset H^1$. The difficulty is the discretization of operator $\mathrm{div}\boldsymbol{\mathrm{div}}$. First we can understand $B : \mathcal{V}_h\to \mathcal{P}_h^{\prime}$ as
\[
\langle B\boldsymbol{\tau}, v\rangle :=b(\boldsymbol{\tau}, v) \quad \forall~ v\in \mathcal{P}_h.
\]
%
Using the Riesz representation induced by the $L^2$-inner product, we can identify $\mathcal{P}_h^{\prime}$ with $\mathcal P_h$ and finally
define $(\mathrm{div}\boldsymbol{\mathrm{div}})_h: \mathcal{V}_h\to \mathcal{P}_h$ as follows: for any $\boldsymbol{\tau}\in\mathcal{V}_h$, $(\mathrm{div}\boldsymbol{\mathrm{div}})_h\boldsymbol{\tau}\in\mathcal{P}_h$ is uniquely determined by
\[
\int_{\Omega}(\mathrm{div}\boldsymbol{\mathrm{div}})_h\boldsymbol{\tau}\, v\,{\rm d}x=b(\boldsymbol{\tau}, v) \quad \forall~v\in\mathcal{P}_h.
\]

To present the commutative diagram,  we need some interpolation operators.
Let $Q_h$ be the $L^2$ orthogonal projection operator from $L^2(\Omega)$ onto $\mathcal{P}_h$ which can be extended to $H^{-1}(\Omega) \to \mathcal P_h$ as $\mathcal P_h \subset H_0^1(\Omega)$.

For any element $K\in \mathcal{T}_h$, define $I_K: H^2(K)\to P_r(K)$ in the following way (cf. \cite{BabuvskaOsbornPitkaranta1980,FalkOsborn1980, Comodi1989,Stenberg1991}) : given $w\in H^2(K)$,  any vertex $a$ of $K$, and any edge $e$ of $K$,
\begin{align*}
I_Kw(a)&=w(a),\\
\int_e(w-I_Kw)v\, {\rm d}s&=0 \quad \forall~v\in P_{r-2}(e),\\
\int_K(w-I_Kw)v\, \, {\rm d}x&=0 \quad \forall~v\in P_{r-3}(K).
\end{align*}
The associated global interpolation operator $I_h$ is given by
\[
(I_h)|_K:=I_K \quad\textrm{ for all } K\in\mathcal{T}_h.
\]
Let $\boldsymbol{I}_K=I_K\times I_K$, $\boldsymbol{I}_h=I_h\times I_h$.

\smallskip
\begin{lemma}\label{lem: Bcomforming}
 $(\mathrm{div}\boldsymbol{\mathrm{div}})_h$ is a conforming discretization of $B$ in the sense that $\ker ((\mathrm{div}\boldsymbol{\mathrm{div}})_h)\subset \ker B$.
\end{lemma}
\begin{proof}
By the definition of $I_h$, we have (cf.~\cite[p.~1058]{BabuvskaOsbornPitkaranta1980})
\begin{equation}\label{eq:temp4}
 b(\boldsymbol \tau_h, v) =  b(\boldsymbol \tau_h, I_hv), \quad \forall \boldsymbol{\tau}_h\in \mathcal{V}_h, v\in \mathcal{P}.
\end{equation}

For any $\boldsymbol{\tau}\in \ker ((\mathrm{div}\boldsymbol{\mathrm{div}})_h)$, we get from \eqref{eq:temp4} that for any $v\in \mathcal{P}$,
\[
\langle B\boldsymbol{\tau}, v\rangle =b(\boldsymbol{\tau}, v)=b(\boldsymbol \tau, I_hv)=\int_{\Omega}(\mathrm{div}\boldsymbol{\mathrm{div}})_h\boldsymbol{\tau}\, I_hv\,{\rm d}x=0.
\]
Thus $\boldsymbol{\tau}\in \ker B$.
\end{proof}

Then define $\boldsymbol{\Pi}_K: \boldsymbol{H}^1(K, \mathbb{S})\to \boldsymbol{P}_{r-1}(K, \mathbb{S})$ in the following way (cf.~\cite{BabuvskaOsbornPitkaranta1980,FalkOsborn1980, Comodi1989,Brezzi.F;Fortin.M1991}): given $\boldsymbol{\tau}\in \boldsymbol{H}^1(K, \mathbb{S})$, for any element $K\in \mathcal{T}_h$ and any edge $e$ of $K$,
\begin{align*}
\int_eM_n\left((\boldsymbol{\tau}-\boldsymbol{\Pi}_K\boldsymbol{\tau})|_K\right)\mu\, {\rm d}s &=0 \quad \forall~\mu\in P_{r-1}(e),\\
\int_K(\boldsymbol{\tau}-\boldsymbol{\Pi}_K\boldsymbol{\tau}):\boldsymbol{\varsigma}\,{\rm d}x &=0 \quad \forall~\boldsymbol{\varsigma}\in \boldsymbol{P}_{r-2}(K, \mathbb{S}).
\end{align*}
The associated global interpolation operator $\boldsymbol{\Pi}_h: \mathcal{V}\to \mathcal{V}_h$ is given by
\[
(\boldsymbol{\Pi}_h)|_K:=\boldsymbol{\Pi}_K \quad\textrm{ for all } K\in\mathcal{T}_h.
\]
From the definition of $\boldsymbol{\Pi}_h$, it holds that
\begin{equation}\label{eq:bpi0}
b(\boldsymbol{\tau}-\boldsymbol{\Pi}_h\boldsymbol{\tau}, v)=0 \quad \forall~\boldsymbol{\tau}\in \mathcal{V}, v\in \mathcal{P}_h.
\end{equation}
Namely $Q_h B = (\mathrm{div}\boldsymbol{\mathrm{div}})_h\boldsymbol\Pi_h$.

\smallskip
\begin{lemma}\label{lem:hhjexactsequence}
The following sequence for the HHJ mixed method
\begin{equation}\label{exactsequence2}
\overline{\boldsymbol{P}}_1(\Omega; \mathbb{R}^2)\autorightarrow{$\subset$}{} \mathcal{S}_h\autorightarrow{$\nabla ^s\times$}{} \mathcal{V}_h \autorightarrow{$(\mathrm{div}\boldsymbol{\mathrm{div}})_h$}{} \mathcal{P}_h\autorightarrow{}{}0
\end{equation}
is an exact sequence.
\end{lemma}
\begin{proof}
As \eqref{exactsequence1}, \eqref{exactsequence2} is a complex by direct computation. Then we prove $\ker((\mathrm{div}\boldsymbol{\mathrm{div}})_h)=\nabla ^s\times\mathcal{S}_h$.
Take any $\boldsymbol{\tau}\in\ker((\mathrm{div}\boldsymbol{\mathrm{div}})_h)$. Since $\ker((\mathrm{div}\boldsymbol{\mathrm{div}})_h)\subset\ker B$ and thus using \eqref{eq:temp8} and the exact sequence \eqref{eq:hhjexactsequencecontinuous} in the continuous level, we find a vector function $\boldsymbol{\phi}\in\boldsymbol{H}^{1}(\Omega; \mathbb{R}^2)$ satisfying $\boldsymbol{\tau}=\nabla ^s\times\boldsymbol{\phi}$. By direct computation, it hold for each $K\in\mathcal{T}_h$
\[
\mathbf{curl}(\mathrm{div}(\boldsymbol{\phi}|_K))=2\mathbf{div}(\boldsymbol{\tau}|_K)\in \boldsymbol{P}_{r-2}(K, \mathbb{R}^2).
\]
Hence $\mathrm{div}(\boldsymbol{\phi}|_K)\in P_{r-1}(K)$, which combined with $\nabla ^s\times\boldsymbol{\phi}=\boldsymbol{\tau}\in  \boldsymbol{P}_{r-1}(K, \mathbb{S})$ means $\boldsymbol{\nabla}(\boldsymbol{\phi}|_K)\in  \boldsymbol{P}_{r-1}(K, \mathbb{S})$.
Therefore $\boldsymbol{\phi}|_K\in\boldsymbol{P}_{r}(K, \mathbb{R}^2)$, i.e. $\boldsymbol{\phi}\in\mathcal{S}_h$.

Using the similar argument as in Lemma~\ref{lem:1}, we have $\ker(\nabla ^s\times)=\overline{\boldsymbol{P}}_1(\Omega; \mathbb{R}^2)$.
To show that \eqref{exactsequence2} is exact, we shall prove $(\mathrm{div}\boldsymbol{\mathrm{div}})_h(\mathcal{V}_h)=\mathcal{P}_h$ by adapting a technique in~\cite[p.~1056]{BabuvskaOsbornPitkaranta1980}.

For any $p\in \mathcal{P}_h$, let $w_h\in \mathcal{P}_h$ be the solution of
\[
\int_{\Omega}\boldsymbol{\nabla}w_h\cdot \boldsymbol{\nabla}v\,{\rm d}x=-\int_{\Omega}pv\, {\rm d}x \quad \forall~v\in\mathcal{P}_h.
\]
Let $\boldsymbol{\sigma}_0=\left(
\begin{array}{cc}
w_h & 0 \\
0 & w_h
\end{array}
\right)$. Thanks to
$M_n(\boldsymbol{\sigma}_0)=\boldsymbol{n}^T\boldsymbol{\sigma}_0\boldsymbol{n}=w_h$ and $w_h\in \mathcal{P}_h$, $\boldsymbol{\sigma}_0\in\mathcal{V}$.
Let $\boldsymbol{\sigma}_I=\boldsymbol{\Pi}_h\boldsymbol{\sigma}_0\in\mathcal{V}_h$. Using \eqref{eq:bpi0}, integration by parts twice,  and the definitions of $\boldsymbol{\sigma}_0$ and $w_h$, it holds for any $v\in\mathcal{P}_h$
\begin{align*}
b(\boldsymbol{\sigma}_I, v)=&\, b(\boldsymbol{\sigma}_0, v)=\sum_{K\in\mathcal{T}_h}\int_{K}\boldsymbol{\sigma}_0:\nabla^2v \,{\rm d}x  - \sum_{K\in\mathcal{T}_h}\int_{\partial K}M_n(\boldsymbol{\sigma}_0)\partial_{\boldsymbol n}v \,{\rm d} s \\
=&\sum_{K\in\mathcal{T}_h}\int_{K}w_h\Delta v \,{\rm d} s - \sum_{K\in\mathcal{T}_h}\int_{\partial K}w_h\partial_{\boldsymbol n}v \,{\rm d} s \\
=& -\int_{\Omega}\boldsymbol{\nabla}w_h\cdot \boldsymbol{\nabla}v \,{\rm d}x=\int_{\Omega}pv \,{\rm d}x,
\end{align*}
from which we can see that $p=(\mathrm{div}\boldsymbol{\mathrm{div}})_h\boldsymbol{\sigma}_I$. The proof is finished.
\end{proof}

\smallskip
\begin{theorem}
We have the following commutative diagram for the HHJ mixed method
\[
\begin{array}{c}
\xymatrix{
  \overline{\boldsymbol{P}}_1(\Omega; \mathbb{R}^2) \ar[r]^{\subset} & \boldsymbol{H}^{1}(\Omega; \mathbb{R}^2) \ar[d]^{\boldsymbol{I}_h} \ar[r]^-{\nabla ^s\times}
                & \boldsymbol{H}^{-1}(\mathrm{div}\boldsymbol{\mathrm{div}},\Omega; \mathbb{S}) \ar[d]^{\boldsymbol{\Pi}_h}   \ar[r]^-{\mathrm{div}\boldsymbol{\mathrm{div}}} & \ar[d]^{Q_h}H^{-1}(\Omega) \ar[r]^{} & 0 \\
 \overline{\boldsymbol{P}}_1(\Omega; \mathbb{R}^2) \ar[r]^{\subset} & \mathcal{S}_h \ar[r]^{\nabla ^s\times}
                & \mathcal{V}_h   \ar[r]^{(\mathrm{div}\boldsymbol{\mathrm{div}})_h} &  \mathcal{P}_h \ar[r]^{}& 0    }
\end{array}.
\]
\end{theorem}
\begin{proof}
The identity $Q_h \mathrm{div}\boldsymbol{\mathrm{div}} = (\mathrm{div}\boldsymbol{\mathrm{div}})_h\boldsymbol \Pi_h$ has been proved in \eqref{eq:bpi0}.

Next we show that for any $\boldsymbol{\phi}\in \boldsymbol{H}^{1}(\Omega; \mathbb{R}^2)\cap {\rm dom}(\boldsymbol I_h)$, $\nabla ^s\times(\boldsymbol{I}_h\boldsymbol{\phi})=\boldsymbol{\Pi}_h\nabla ^s\times\boldsymbol{\phi}$.
For each $\boldsymbol{\varsigma}\in \boldsymbol{P}_{r-2}(K, \mathbb{S})$ and $K\in\mathcal{T}_h$, it follows from integration by parts and the definitions of $\boldsymbol{\Pi}_h$ and $\boldsymbol{I}_h$
\begin{equation}\label{eq:temp1}
\int_K(\nabla ^s\times(\boldsymbol{I}_h\boldsymbol{\phi})-\boldsymbol{\Pi}_h(\nabla ^s\times\boldsymbol{\phi})):\boldsymbol{\varsigma}\, {\rm d}x=\int_K\nabla ^s\times(\boldsymbol{I}_h\boldsymbol{\phi}-\boldsymbol{\phi}):\boldsymbol{\varsigma}\, {\rm d}x=0.
\end{equation}
On each $e\in \mathcal{E}_h(K)$, by the definition of $\boldsymbol{\Pi}_h$, it holds for any $\mu\in P_{r-1}(e)$
\[
\int_eM_n(\nabla ^s\times(\boldsymbol{I}_h\boldsymbol{\phi})-\boldsymbol{\Pi}_h(\nabla ^s\times\boldsymbol{\phi}))\mu ~\textrm{d}s=\int_eM_n(\nabla ^s\times(\boldsymbol{I}_h\boldsymbol{\phi}-\boldsymbol{\phi}))\mu~\textrm{d}s.
\]
Note the fact that $M_n(\nabla ^s\times(\boldsymbol{I}_h\boldsymbol{\phi}-\boldsymbol{\phi}))=\partial_{\boldsymbol{t}}((\boldsymbol{I}_h\boldsymbol{\phi}-\boldsymbol{\phi})\cdot\boldsymbol{n})$.
Hence we get from integration by parts and the definition of $\boldsymbol{I}_h$
\begin{align}
\int_eM_n(\nabla ^s\times(\boldsymbol{I}_h\boldsymbol{\phi})-\boldsymbol{\Pi}_h(\nabla ^s\times\boldsymbol{\phi}))\mu ~\textrm{d}s=& \int_e\partial_{\boldsymbol{t}}((\boldsymbol{I}_h\boldsymbol{\phi}-\boldsymbol{\phi})\cdot\boldsymbol{n})\mu ~\textrm{d}s=0. \label{eq:temp2}
\end{align}
Since $(\nabla ^s\times(\boldsymbol{I}_h\boldsymbol{\phi})-\boldsymbol{\Pi}_h(\nabla ^s\times\boldsymbol{\phi}))|_K\in \boldsymbol{P}_{k-1}(K, \mathbb{S})$,  \eqref{eq:temp1}-\eqref{eq:temp2} together with the wellposedness of $\boldsymbol{\Pi}_h$ means $\nabla ^s\times(\boldsymbol{I}_h\boldsymbol{\phi})-\boldsymbol{\Pi}_h(\nabla ^s\times\boldsymbol{\phi})=\boldsymbol{0}$, i.e. $\nabla ^s\times(\boldsymbol{I}_h\boldsymbol{\phi})=\boldsymbol{\Pi}_h(\nabla ^s\times\boldsymbol{\phi})$.
\end{proof}

\smallskip
\begin{remark}\rm
It is worth mentioning that we use the natural Sobolev spaces with minimal regularity in the top sequence of the commutative diagram. The interpolation operators $\boldsymbol I_h$ and $\boldsymbol \Pi_h$, however, are defined for smoother functions and not bounded in the corresponding Sobolev norms. Namely we treat these interpolation operators as densely defined unbounded operators. It is possible to use the smoothing procedure  \cite{ArnoldFalkWinther2006} to define stable quasi-interpolation operators while preserving the commutative property. $\Box$
\end{remark}

\section{Stable Decomposition and Strengthened Cauchy Schwarz Inequality}\label{section:sdscs}
In this section, we will present a stable decomposition for the space $\mathcal V_h$ used in the HHJ mixed method. We assume that there exists  a sequence of meshes $\mathcal T_1, \mathcal T_2, \ldots, \mathcal T_J = \mathcal T_h$. Hereafter subscript $k$ is used to indicate spaces associated to triangulation $\mathcal{T}_k$. The triangulation $\mathcal T_1$ is a shape regular triangulation of $\Omega$ and $\mathcal T_{k+1}$ is obtained by dividing each triangle  in $\mathcal T_{k}$ into four congruent small triangles. The mesh size of $\mathcal T_k$ will be denoted by $h_k$. By the construction, the ratio $\gamma^2 = h_{k+1}/h_k = 1/2$.

Based on the exact sequence \eqref{exactsequence2}, define $\mathcal{K}_k:=\nabla ^s\times\mathcal{S}_k$ for $k=1,2,\cdots, J$. Obviously we have the following macro-decomposition
\[
\mathcal{K}_h=\mathcal{K}_1+\mathcal{K}_2+\cdots+\mathcal{K}_J.
\]

Denote by $N_k$ the number of vertices in $\mathcal{T}_k$ for $k=1,2, \cdots, J$.
Define the $i$-th patch $\omega_{k,i}$ in the $k$-th level as the union of the elements sharing the common $i$-th vertex in $\mathcal{T}_k$
for $i=1,2,\cdots, N_k$. Let
\[
\mathcal{S}_{k,i}:=\{\boldsymbol{\phi}\in\mathcal{S}_k: \textrm{supp}(\boldsymbol{\phi})\subset \omega_{k,i}\}, \; \mathcal{V}_{k,i}:=\{\boldsymbol{\tau}\in\mathcal{V}_k: \textrm{supp}(\boldsymbol{\tau})\subset \omega_{k,i}\},
\]
and $\mathcal{K}_{k,i}:=\nabla ^s\times\mathcal{S}_{k,i}$. It can be verified that
$$\mathcal{S}_h=\sum_{k=1}^J\mathcal{S}_k=\sum_{k=1}^J\sum_{i=1}^{N_k}\mathcal{S}_{k,i}, \quad \mathcal{V}_h=\sum_{k=1}^J\mathcal{V}_k=\sum_{k=1}^J\sum_{i=1}^{N_k}\mathcal{V}_{k,i},
$$
\begin{equation}\label{eq:Kdec}
\mathcal{K}_h=\sum_{k=1}^J\mathcal{K}_k=\sum_{k=1}^J\sum_{i=1}^{N_k}\mathcal{K}_{k,i}.
\end{equation}
We shall prove the space decomposition \eqref{eq:Kdec} is stable in the energy norm introduced by $\nabla ^s\times$.

\subsection{Equivalent norms}
We first introduce the following quotient spaces
\begin{align*}
&\widetilde{\mathcal{S}}:=\left\{\boldsymbol{\phi}\in \boldsymbol{H}^1(\Omega; \mathbb{R}^2): \int_{\Omega}\boldsymbol{\phi}\, {\rm d}x=\boldsymbol{0}, \quad \int_{\Omega}\boldsymbol{\phi}\cdot\boldsymbol{x}\, {\rm d}x=0 \right\}, \\
&\widetilde{\mathcal{S}}_k:=\left\{\boldsymbol{\phi}\in \mathcal{S}_k: \int_{\Omega}\boldsymbol{\phi}\, {\rm d}x=\boldsymbol{0}, \quad \int_{\Omega}\boldsymbol{\phi}\cdot\boldsymbol{x}\, {\rm d}x=0 \right\}.
\end{align*}
It is easy to see that
\[
\boldsymbol{H}^1(\Omega; \mathbb{R}^2)=\widetilde{\mathcal{S}}\oplus\overline{\boldsymbol{P}}_1(\Omega; \mathbb{R}^2), \quad \mathcal{S}_k=\widetilde{\mathcal{S}}_k\oplus\overline{\boldsymbol{P}}_1(\Omega; \mathbb{R}^2).
\]
Notation $\oplus$ means the direct sum. Since the polynomials of degree less than or equal to $1$ belong to the space $\mathcal S_k$, the spaces $\widetilde{\mathcal{S}}_k$ are nested.
Let
\[
\widetilde{\mathcal{W}}:=\{\boldsymbol{\phi}\in \boldsymbol{H}^1(\Omega; \mathbb{R}^2): \int_{\Omega}\boldsymbol{\phi}\cdot\boldsymbol{\psi}\,{\rm d}x=0\quad\forall\,\boldsymbol{\psi}\in\overline{\boldsymbol{P}}_1^{\textrm{Rot}}(\Omega; \mathbb{R}^2)\}, \quad
\widetilde{\mathcal{W}}_k:=\widetilde{\mathcal{W}}\cap\mathcal{S}_k.
\]
It is obvious that $\boldsymbol{\phi}^{\perp}\in\widetilde{\mathcal{S}}$ if $\boldsymbol{\phi}\in\widetilde{\mathcal{W}}$, and vice versa.

The following lemma says that in the quotient space $\widetilde{\mathcal{S}}$, the differential operator $\nabla ^s\times$ introduces a norm equivalent to $H^1$ norm. Similar result has been proved in~\cite{CarstensenGallistlHu2014} on a slightly different quotient space.

\begin{lemma}
It holds
\begin{equation}\label{eq:Jkorn1}
\|\boldsymbol{\phi}\|_1\lesssim \|\nabla ^s\times\boldsymbol{\phi}\|_0 \quad \forall~\boldsymbol{\phi}\in \widetilde{\mathcal{S}}.
\end{equation}
\end{lemma}
\begin{proof}
By a direct computation, we have for any vector $\boldsymbol{\phi}$ and $\boldsymbol{\psi}$
\begin{equation}\label{eq:temp3}
\nabla ^s\times\boldsymbol{\phi} : \nabla ^s\times\boldsymbol{\psi}=\boldsymbol{\varepsilon}(\boldsymbol{\phi}^{\perp}):\boldsymbol{\varepsilon}(\boldsymbol{\psi}^{\perp}).
\end{equation}
Since $\boldsymbol{\phi}\in \widetilde{\mathcal{S}}$, we have $\boldsymbol{\phi}^{\perp}\in\widetilde{\mathcal{W}}$.
According to the Korn's inequality (see (2.2) in~\cite{Brenner1994} and Theorem 2.3 in~\cite{Ciarlet2010}), it follows
\[
\|\boldsymbol{\phi}^{\perp}\|_1\lesssim \|\boldsymbol{\varepsilon}(\boldsymbol{\phi}^{\perp})\|_0.
\]
Then we obtain from \eqref{eq:temp3}
\[
\|\boldsymbol{\phi}\|_1=\|\boldsymbol{\phi}^{\perp}\|_1\lesssim \|\boldsymbol{\varepsilon}(\boldsymbol{\phi}^{\perp})\|_0=\|\nabla ^s\times\boldsymbol{\phi}\|_0,
\]
which ends the proof.
\end{proof}

\subsection{Strengthened Cauchy Schwarz Inequality}
Thanks to the relation \eqref{eq:temp3}, the following strengthened Cauchy Schwarz (SCS) inequality can be proved using the technique for the scalar case; see Xu~\cite{Xu1992}.
\begin{lemma}\label{lem:temp1}
Let $1\leq k\leq l\leq J$. We have
\[
\int_{\Omega}\nabla ^s\times \boldsymbol{\phi}: \nabla ^s\times \boldsymbol{\psi}\,{\rm d}x\lesssim \gamma^{l-k}h_l^{-1} \|\nabla ^s\times\boldsymbol{\phi}\|_0\|\boldsymbol{\psi}\|_0 \quad \forall~\boldsymbol{\phi}\in \mathcal{S}_k, \boldsymbol{\psi}\in \mathcal{S}_l.
\]
\end{lemma}
\begin{proof}
For any $K\in\mathcal{T}_k$, we get from integration by parts and the Cauchy-Swarchz inequality
\begin{align*}
\int_{K}\nabla ^s\times\boldsymbol{\phi}: \nabla ^s\times\boldsymbol{\psi}\,{\rm d}x = & \int_{K}\nabla ^s\times\boldsymbol{\phi}: \mathbf{Curl}\boldsymbol{\psi}\,{\rm d}x \\
=& \int_{K}\mathbf{rot}(\nabla ^s\times\boldsymbol{\phi})\cdot \boldsymbol{\psi}\,{\rm d}x -\int_{\partial K}((\nabla ^s\times\boldsymbol{\phi})\boldsymbol{t})\cdot\boldsymbol{\psi} \,{\rm d} s \\
\lesssim & \|\mathbf{rot}(\nabla ^s\times\boldsymbol{\phi})\|_{0,K}\|\boldsymbol{\psi}\|_{0,K} + \|\nabla ^s\times\boldsymbol{\phi}\|_{0,\partial K}\|\boldsymbol{\psi}\|_{0,\partial K}.
\end{align*}
By the inverse inequality, it holds
\[
\|\boldsymbol{\psi}\|_{0,\partial K}^2\leq \sum_{\widetilde K\in\mathcal{T}_l, \widetilde K\subset K}\|\boldsymbol{\psi}\|_{0,\partial \widetilde K}^2\lesssim h_l^{-1}\sum_{\widetilde K\in\mathcal{T}_l, \widetilde K\subset K}\|\boldsymbol{\psi}\|_{0, \widetilde K}^2=h_l^{-1}\|\boldsymbol{\psi}\|_{0, K}^2.
\]
Then we get from the last two inequalities and the inverse inequality
\begin{align*}
\int_{K}\nabla ^s\times\boldsymbol{\phi}: \nabla ^s\times\boldsymbol{\psi}\,{\rm d}x
\lesssim & h_k^{-1}\|\nabla ^s\times\boldsymbol{\phi}\|_{0,K}\|\boldsymbol{\psi}\|_{0,K} + (h_kh_l)^{-1/2}\|\nabla ^s\times\boldsymbol{\phi}\|_{0,K}\|\boldsymbol{\psi}\|_{0,K} \\
\lesssim & (h_kh_l)^{-1/2}\|\nabla ^s\times\boldsymbol{\phi}\|_{0,K}\|\boldsymbol{\psi}\|_{0,K} \\
= & \gamma^{l-k}h_l^{-1}\|\nabla ^s\times\boldsymbol{\phi}\|_{0,K}\|\boldsymbol{\psi}\|_{0,K}.
\end{align*}
Due to the Cauchy-Swarchz inequality, we obtain
\begin{align*}
&\int_{\Omega}\nabla ^s\times\boldsymbol{\phi}: \nabla ^s\times\boldsymbol{\psi}\,{\rm d}x=\sum_{K\in\mathcal{T}_k}\int_{K}\nabla ^s\times\boldsymbol{\phi}: \nabla ^s\times\boldsymbol{\psi}\,{\rm d}x \\
\lesssim & \gamma^{l-k}h_l^{-1}\sum_{K\in\mathcal{T}_k}\|\nabla ^s\times\boldsymbol{\phi}\|_{0,K}\|\boldsymbol{\psi}\|_{0,K}
\lesssim \gamma^{l-k}h_l^{-1} \|\nabla ^s\times\boldsymbol{\phi}\|_0\|\boldsymbol{\psi}\|_0,
\end{align*}
as required.
\end{proof}

Next we prove the SCS inequality for the space decomposition \eqref{eq:Kdec} of $\mathcal K$. For this, we use the lexicographical order of the double index, i.e., $(l,j)>(k,i)$ if $l>k$ or $l=k$, $j>i$.
\begin{theorem}[SCS]
For any $\boldsymbol{\tau}_{k,i}\in\mathcal{K}_{k,i}$ and $\boldsymbol{\varsigma}_{l,j}\in\mathcal{K}_{l,j}$, we have
\[
\sum_{k=1}^J\sum_{i=1}^{N_k}\sum_{(l,j)>(k,i)}\int_{\Omega}\boldsymbol{\tau}_{k,i}:\boldsymbol{\varsigma}_{l,j}\,{\rm d}x \lesssim \left(\sum_{k=1}^J\sum_{i=1}^{N_k}\|\boldsymbol{\tau}_{k,i}\|_0^2\right)^{1/2}\left(\sum_{l=1}^J\sum_{j=1}^{N_l}\|\boldsymbol{\varsigma}_{l,j}\|_0^2\right)^{1/2}.
\]
\end{theorem}

\begin{proof}
Let $\boldsymbol{\tau}_{k,i}=\nabla ^s\times\boldsymbol{\phi}_{k,i}$ and $\boldsymbol{\varsigma}_{l,j}=\nabla ^s\times\boldsymbol{\psi}_{l,j}$ with $\boldsymbol{\phi}_{k,i}\in\mathcal{S}_{k,i}$ and $\boldsymbol{\psi}_{l,j}\in\mathcal{S}_{l,j}$.
Set
$
\boldsymbol{\phi}_{k}=\sum\limits_{i=1}^{N_k}\boldsymbol{\phi}_{k,i}$ and $\boldsymbol{\psi}_{l}=\sum\limits_{j=1}^{N_l}\boldsymbol{\psi}_{l,j}.
$
Using Lemma~\ref{lem:temp1} and the fact that $h_l^{-1}\|\boldsymbol{\psi}_{l,j}\|_0\eqsim \|\nabla ^s\times\boldsymbol{\psi}_{l,j}\|_0$, we get
\begin{align*}
&\sum_{k=1}^J\sum_{i=1}^{N_k}\sum_{l>k}\sum_{j=1}^{N_l}\int_{\Omega}\boldsymbol{\tau}_{k,i}:\boldsymbol{\varsigma}_{l,j}\,{\rm d}x \\
=& \sum_{k=1}^J\sum_{i=1}^{N_k}\sum_{l>k}\sum_{j=1}^{N_l}\int_{\Omega}\nabla ^s\times\boldsymbol{\phi}_{k,i} : \nabla ^s\times \boldsymbol{\psi}_{l,j}\,{\rm d}x
= \sum_{k=1}^J\sum_{l>k}\int_{\Omega}\nabla ^s\times\boldsymbol{\phi}_{k} : \nabla ^s\times \boldsymbol{\psi}_{l}\,{\rm d}x\\
\lesssim & \sum_{k=1}^J\sum_{l>k}\gamma^{l-k}h_l^{-1} \|\nabla ^s\times \boldsymbol{\phi}_{k}\|_0\|\boldsymbol{\psi}_{l}\|_0
\lesssim  \left(\sum_{k=1}^J\|\nabla ^s\times \boldsymbol{\phi}_{k}\|_0^2\right)^{1/2}\left(\sum_{l=1}^Jh_l^{-2}\|\boldsymbol{\psi}_{l}\|_0^2\right)^{1/2} \\
\lesssim & \left(\sum_{k=1}^J\sum_{i=1}^{N_k}\|\nabla ^s\times\boldsymbol{\phi}_{k,i}\|_0^2\right)^{1/2}\left(\sum_{l=1}^J\sum_{j=1}^{N_l}h_l^{-2}\|\boldsymbol{\psi}_{l,j}\|_0^2\right)^{1/2} \\
\lesssim & \left(\sum_{k=1}^J\sum_{i=1}^{N_k}\|\boldsymbol{\tau}_{k,i}\|_0^2\right)^{1/2}\left(\sum_{l=1}^J\sum_{j=1}^{N_l}\|\boldsymbol{\varsigma}_{l,j}\|_0^2\right)^{1/2}.
\end{align*}
On the other hand, since the index set $n_k(i):=\{j\in\{i+1,\cdots, N_k\}, \omega_{k,i}\cap\omega_{k,j}\neq\emptyset\}$ is finite in the $k$th level,
\begin{align*}
\sum_{k=1}^J\sum_{i=1}^{N_k}\sum_{j=i+1}^{N_k}\int_{\Omega}\boldsymbol{\tau}_{k,i}:\boldsymbol{\varsigma}_{k,j}\,{\rm d}x =& \sum_{k=1}^J\sum_{i=1}^{N_k}\sum_{j\in n_k(i)}\int_{\Omega}\boldsymbol{\tau}_{k,i}:\boldsymbol{\varsigma}_{k,j}\,{\rm d}x  \\
\lesssim & \left(\sum_{k=1}^J\sum_{i=1}^{N_k}\|\boldsymbol{\tau}_{k,i}\|_0^2\right)^{1/2}\left(\sum_{l=1}^J\sum_{j=1}^{N_l}\|\boldsymbol{\varsigma}_{l,j}\|_0^2\right)^{1/2}.
\end{align*}
The summation of the last two inequalities implies the desired result.
\end{proof}

\subsection{Stable Decomposition}
Let $\boldsymbol{Q}_k$ be the $L^2$ projection from $\boldsymbol{L}^2(\Omega; \mathbb{R}^2)$ onto $\mathcal{S}_k$. It is easy to see that $\boldsymbol{Q}_k\boldsymbol{\phi}\in\widetilde{\mathcal{S}}_k$ if $\boldsymbol{\phi}\in\widetilde{\mathcal{S}}$. Due to the nestedness of spaces $\mathcal S_k$, we also have $\boldsymbol{Q}_k\boldsymbol{Q}_l = \boldsymbol{Q}_k$ for $l\geq k$. The following first order error estimate of $\boldsymbol{Q}_k$ is well known
\begin{equation}
\| (I - \boldsymbol{Q}_k)\boldsymbol{\psi}\|_0\lesssim h_k\|\boldsymbol{\psi}\|_1,\quad \text{for all } \boldsymbol{\psi}\in \boldsymbol{H}^1(\Omega; \mathbb{R}^2).
\end{equation}

\smallskip

\begin{lemma}\label{lem:temp2}
Let
$
\boldsymbol{W}_i=(\boldsymbol{Q}_i-\boldsymbol{Q}_{i-1})\widetilde{\mathcal{S}}_h
$
for $i=1,2, \cdots, J$. We have
\[
\int_{\Omega}\nabla ^s\times\boldsymbol{\phi}: \nabla ^s\times\boldsymbol{\psi}\,{\rm d}x\lesssim \gamma^{|i-j|} \|\nabla ^s\times\boldsymbol{\phi}\|_0\|\nabla ^s\times\boldsymbol{\psi}\|_0
\]
for any $\boldsymbol{\phi}\in \boldsymbol{W}_i$ and $\boldsymbol{\psi}\in \boldsymbol{W}_j$.
\end{lemma}
\begin{proof}
According to the estimate of $\boldsymbol{Q}_{j-1}$ and \eqref{eq:Jkorn1},
\[
\|\boldsymbol{\psi}\|_0=\|(\boldsymbol{I}-\boldsymbol{Q}_{j-1})\boldsymbol{\psi}\|_0\lesssim h_j\|\boldsymbol{\psi}\|_1\lesssim h_j\|\nabla ^s\times\boldsymbol{\psi}\|_0 \quad \forall~\boldsymbol{\psi}\in \boldsymbol{W}_j.
\]
The proof is finished from Lemma~\ref{lem:temp1}.
\end{proof}

Let $\boldsymbol{P}_k$ be the $\nabla ^s\times$-orthogonal projection onto $\widetilde{\mathcal{S}}_k$, that is for any $\boldsymbol{\phi}\in\widetilde{\mathcal{S}}$,
\begin{equation}\label{eq:operatorP}
\int_{\Omega}\nabla ^s\times(\boldsymbol{P}_k\boldsymbol{\phi}): \nabla ^s\times\boldsymbol{\chi}\,{\rm d}x=\int_{\Omega}\nabla ^s\times\boldsymbol{\phi}: \nabla ^s\times\boldsymbol{\chi}\,{\rm d}x \quad \forall~\boldsymbol{\chi}\in\widetilde{\mathcal{S}}_k.
\end{equation}
To derive the error estimate of $\boldsymbol{P}_k$, we introduce another operator $\boldsymbol{R}_k$ which is related to the pure traction problem in the planar linear elasticity.
Let $\boldsymbol{R}_k: \widetilde{\mathcal{W}}\to\widetilde{\mathcal{W}}_k$ be defined as follows: for any $\boldsymbol{\phi}\in\widetilde{\mathcal{W}}$, $\boldsymbol{R}_k\boldsymbol{\phi}$ is uniquely determined by
\[
\int_{\Omega}\boldsymbol{\varepsilon}(\boldsymbol{R}_k\boldsymbol{\phi}): \boldsymbol{\varepsilon}(\boldsymbol{\chi})\, {\rm d}x=\int_{\Omega}\boldsymbol{\varepsilon}(\boldsymbol{\phi}): \boldsymbol{\varepsilon}(\boldsymbol{\chi})\,{\rm d}x \quad \forall~\boldsymbol{\chi}\in\widetilde{\mathcal{W}}_k.
\]
According to the standard finite element approximation theory (cf.~\cite[(5.9)]{BramblePasciak1987}), we have
\begin{equation}\label{eq:estimateFEMelas}
\|\boldsymbol{\phi}-\boldsymbol{R}_k\boldsymbol{\phi}\|_{1-\alpha}\lesssim h_k^{\alpha}\|\boldsymbol{\phi}\|_1\quad \forall~\boldsymbol{\phi}\in\widetilde{\mathcal{W}}
\end{equation}
for some constant $\alpha\in(0,1]$. Here $\alpha$ is the parameter indicating the elliptic regularity of the pure traction problem in the planar linear elasticity defined in $\Omega$ (cf. \cite{Grisvard1992}). $ \alpha= 1$ if $\Omega$ is convex and
$0<\alpha < 1$ if $\Omega$ is nonconvex.

\begin{lemma}
It holds
\begin{equation}\label{eq:estimateFEM}
\|\boldsymbol{\phi}-\boldsymbol{P}_k\boldsymbol{\phi}\|_{1-\alpha}\lesssim h_k^{\alpha}\|\boldsymbol{\phi}\|_1\quad \forall~\boldsymbol{\phi}\in\widetilde{\mathcal{S}}.
\end{equation}
\end{lemma}
\begin{proof}
Due to \eqref{eq:temp3}, \eqref{eq:operatorP} is equivalent to
\[
\int_{\Omega}\boldsymbol{\varepsilon}((\boldsymbol{P}_k\boldsymbol{\phi})^{\perp}): \boldsymbol{\varepsilon}(\boldsymbol{\chi}^{\perp})\, {\rm d}x=\int_{\Omega}\boldsymbol{\varepsilon}(\boldsymbol{\phi}^{\perp}): \boldsymbol{\varepsilon}(\boldsymbol{\chi}^{\perp})\,{\rm d}x \quad \forall~\boldsymbol{\chi}\in\widetilde{\mathcal{S}}_k,
\]
which is nothing but
\[
\int_{\Omega}\boldsymbol{\varepsilon}((\boldsymbol{P}_k\boldsymbol{\phi})^{\perp}): \boldsymbol{\varepsilon}(\boldsymbol{\chi})\, {\rm d}x=\int_{\Omega}\boldsymbol{\varepsilon}(\boldsymbol{\phi}^{\perp}): \boldsymbol{\varepsilon}(\boldsymbol{\chi})\,{\rm d}x \quad \forall~\boldsymbol{\chi}\in\widetilde{\mathcal{W}}_k.
\]
Noting the fact that $\boldsymbol{\phi}^{\perp}\in\widetilde{\mathcal{W}}$ and $(\boldsymbol{P}_k\boldsymbol{\phi})^{\perp}\in\widetilde{\mathcal{W}}_k$,
we get $(\boldsymbol{P}_k\boldsymbol{\phi})^{\perp}=\boldsymbol{R}_k(\boldsymbol{\phi}^{\perp})$. Therefore it follows from \eqref{eq:estimateFEMelas}
\begin{align*}
\|\boldsymbol{\phi}-\boldsymbol{P}_k\boldsymbol{\phi}\|_{1-\alpha}=&\|\boldsymbol{\phi}^{\perp}-(\boldsymbol{P}_k\boldsymbol{\phi})^{\perp}\|_{1-\alpha}
=\|\boldsymbol{\phi}^{\perp}-\boldsymbol{R}_k(\boldsymbol{\phi}^{\perp})\|_{1-\alpha} \\
\lesssim & h_k^{\alpha}\|\boldsymbol{\phi}^{\perp}\|_1=h_k^{\alpha}\|\boldsymbol{\phi}\|_1,
\end{align*}
as required.
\end{proof}

Again using the technique for the scalar $H^1$ space~\cite{Xu1992}, we have the following stable decomposition of functions in $\tilde S_h$.
\begin{lemma}[Stable macro-decomposition]\label{lem:stablemacrodscomposition}
For each $\boldsymbol{\phi}\in \widetilde{\mathcal{S}}_h$, there exist $\boldsymbol{\phi}_k\in \widetilde{\mathcal{S}}_k$, $k=1,2,\cdots, J$ such that
\[
\boldsymbol{\phi}=\sum_{k=1}^J\boldsymbol{\phi}_k, \;\textrm{ and }\; \sum_{k=1}^J\|\nabla ^s\times\boldsymbol{\phi}_k\|_0^2\eqsim \|\nabla ^s\times\boldsymbol{\phi}\|_0^2.
\]
\end{lemma}

\begin{proof}
Let $\tilde{\boldsymbol{Q}}_k=\boldsymbol{Q}_k-\boldsymbol{Q}_{k-1}$, $\boldsymbol{\phi}_k=\tilde{\boldsymbol{Q}}_k\boldsymbol{\phi}$ and $\boldsymbol{\psi}_i=(\boldsymbol{P}_i-\boldsymbol{P}_{i-1})\boldsymbol{\phi}$ for $i, k=1,2,\cdots, J$. Using Cauchy-Swarchz inequality, it holds
\begin{align*}
\sum_{k=1}^J\|\nabla ^s\times\boldsymbol{\phi}_k\|_0^2=&\sum_{k=1}^J\|\nabla ^s\times(\tilde{\boldsymbol{Q}}_k\boldsymbol{\phi})\|_0^2 \\
=&\sum_{k=1}^J\sum_{i,j=k}^J\int_{\Omega}\nabla ^s\times(\tilde{\boldsymbol{Q}}_k\boldsymbol{\psi}_i):\nabla ^s\times(\tilde{\boldsymbol{Q}}_k\boldsymbol{\psi}_j)\, {\rm d}x \\
=&\sum_{i,j=1}^J\sum_{k=1}^{i\wedge j}\int_{\Omega}\nabla ^s\times(\tilde{\boldsymbol{Q}}_k\boldsymbol{\psi}_i):\nabla ^s\times(\tilde{\boldsymbol{Q}}_k\boldsymbol{\psi}_j)\, {\rm d}x \\
\leq & \sum_{i,j=1}^J\sum_{k=1}^{i\wedge j}\|\nabla ^s\times(\tilde{\boldsymbol{Q}}_k\boldsymbol{\psi}_i)\|_0\|\nabla ^s\times(\tilde{\boldsymbol{Q}}_k\boldsymbol{\psi}_j)\|_0,
\end{align*}
where $i\wedge j=\min\{i,j\}$.
According to the inverse inequality, the error estimate of $\boldsymbol{Q}_k$, and \eqref{eq:estimateFEM}, we have
\[
\|\nabla ^s\times(\tilde{\boldsymbol{Q}}_k\boldsymbol{\psi}_i)\|_0\lesssim |\tilde{\boldsymbol{Q}}_k\boldsymbol{\psi}_i|_1\lesssim h_k^{-\alpha}\|\tilde{\boldsymbol{Q}}_k\boldsymbol{\psi}_i\|_{1-\alpha}\lesssim h_k^{-\alpha}\|\boldsymbol{\psi}_i\|_{1-\alpha}\lesssim h_k^{-\alpha}h_i^{\alpha}\|\boldsymbol{\psi}_i\|_{1}.
\]
Combining the last two inequalities, we get from the strengthened Cauchy-Swarchz inequality and \eqref{eq:Jkorn1}
\begin{align*}
\sum_{k=1}^J\|\nabla ^s\times\boldsymbol{\phi}_k\|_0^2\lesssim & \sum_{i,j=1}^J\sum_{k=1}^{i\wedge j}h_k^{-2\alpha}h_j^{\alpha}h_i^{\alpha}\|\boldsymbol{\psi}_i\|_{1}\|\boldsymbol{\psi}_j\|_{1} \\
\lesssim & \sum_{i,j=1}^Jh_{i\wedge j}^{-2\alpha}h_j^{\alpha}h_i^{\alpha}\|\boldsymbol{\psi}_i\|_{1}\|\boldsymbol{\psi}_j\|_{1}
\lesssim \sum_{i,j=1}^J\gamma^{\alpha|i-j|}\|\boldsymbol{\psi}_i\|_{1}\|\boldsymbol{\psi}_j\|_{1} \\
\lesssim & \sum_{i=1}^J \|\boldsymbol{\psi}_i\|_{1}^2 \lesssim \sum_{i=1}^J \|\nabla ^s\times\boldsymbol{\psi}_i\|_{0}^2=\|\nabla ^s\times\boldsymbol{\phi}\|_{0}^2.
\end{align*}

On the other side, it follows from Lemma~\ref{lem:temp2} and the strengthened Cauchy-Swarchz inequality
\begin{align*}
\|\nabla ^s\times\boldsymbol{\phi}\|_0^2 = & \sum_{i,j=1}^J\int_{\Omega}\nabla ^s\times(\tilde{\boldsymbol{Q}}_i\boldsymbol{\phi}):\nabla ^s\times(\tilde{\boldsymbol{Q}}_j\boldsymbol{\phi})\, {\rm d}x \\
\lesssim & \sum_{i,j=1}^J\gamma^{|i-j|} \|\nabla ^s\times(\tilde{\boldsymbol{Q}}_i\boldsymbol{\phi})\|_0\|\nabla ^s\times(\tilde{\boldsymbol{Q}}_j\boldsymbol{\phi})\|_0 \\
\lesssim & \sum_{i=1}^J\|\nabla ^s\times(\tilde{\boldsymbol{Q}}_i\boldsymbol{\phi})\|_0 =\sum_{i=1}^J\|\nabla ^s\times\boldsymbol{\phi}_i\|_0.
\end{align*}
The proof is completed.
\end{proof}

\begin{lemma}[Stable micro-decomposition]\label{lem:stablemicrodscomposition}
Let $\boldsymbol{\phi}_k=(\boldsymbol{Q}_k-\boldsymbol{Q}_{k-1})\boldsymbol{\phi}$ with $\boldsymbol{\phi}\in \widetilde{\mathcal{S}}_h$. Then based on the decomposition \eqref{eq:Kdec}, there exists $\boldsymbol{\phi}_{k,i}\in \mathcal{S}_{k,i}$, $i=1,2,\cdots, N_k$ such that
\[
\boldsymbol{\phi}_k=\sum_{i=1}^{N_k}\boldsymbol{\phi}_{k,i}, \;\textrm{ and }\; \sum_{i=1}^{N_k}\|\nabla ^s\times\boldsymbol{\phi}_{k,i}\|_0^2\lesssim \|\nabla ^s\times\boldsymbol{\phi}_k\|_0^2.
\]
\end{lemma}
\begin{proof}
Let $\boldsymbol \phi_{k} = \sum _{j=1}^{N_k}\boldsymbol \phi_{k,j}$ be a decomposition such that $\supp \boldsymbol\phi_{k,j}\in \omega_{k, j}$. Such a decomposition can be obtained by partitioning the nodal basis decomposition of $\boldsymbol \phi_{k}$. For example, for a basis function associated to an edge, it can be split as half and half to the patch of each endpoint of this edge.

According to the inverse inequality and the stability of the basis decomposition in $L^2$-norm, we have
\[
\sum_{i=1}^{N_k}\|\nabla ^s\times\boldsymbol{\phi}_{k,i}\|_0^2\lesssim h_k^{-2}\sum_{i=1}^{N_k}\|\boldsymbol{\phi}_{k,i}\|_0^2\lesssim h_k^{-2}\|\boldsymbol{\phi}_{k}\|_0^2.
\]
Since $\boldsymbol{\phi}_{k}=(\boldsymbol{I}-\boldsymbol{Q}_{k-1})\boldsymbol{\phi}_{k}$, it holds from the estimate of $\boldsymbol{Q}_{k-1}$ and \eqref{eq:Jkorn1}
\[
\|\boldsymbol{\phi}_{k}\|_0=\|(\boldsymbol{I}-\boldsymbol{Q}_{k-1})\boldsymbol{\phi}_{k}\|_0\lesssim h_k\|\boldsymbol{\phi}_{k}\|_1\lesssim h_k \|\nabla ^s\times\boldsymbol{\phi}_k\|_0.
\]
Therefore we can finish the proof from the last two inequalities. \end{proof}

Hence the following multilevel stable decomposition of $\mathcal K_h$ can be derived by the combination of Lemmas~\ref{lem:stablemacrodscomposition}-\ref{lem:stablemicrodscomposition}.
\begin{theorem}[Stable decomposition]
For each $\boldsymbol{\sigma}\in\mathcal{K}_h$, there exists $\boldsymbol{\sigma}_{k,i}\in\mathcal{K}_{k,i}$, $k=1,2,\cdots, J$, $i=1,2,\cdots, N_k$
such that $\boldsymbol{\sigma}$
\[
\boldsymbol{\sigma}=\sum_{k=1}^{J}\sum_{i=1}^{N_k}\boldsymbol{\sigma}_{k,i} \;\textrm{ and }\; \sum_{k=1}^{J}\sum_{i=1}^{N_k}\|\boldsymbol{\sigma}_{k,i}\|_0^2\lesssim \|\boldsymbol{\sigma}\|_0^2.
\]
\end{theorem}
\begin{proof}
Since $\boldsymbol{\sigma}\in\mathcal{K}_h$, we can find a unique element $\boldsymbol \phi\in \widetilde{\mathcal S}_h$ such that $\boldsymbol{\sigma}=\nabla ^s\times\boldsymbol{\phi}$. Let $\boldsymbol{\phi}_k=(\boldsymbol{Q}_k-\boldsymbol{Q}_{k-1})\boldsymbol{\phi}$.
We get from Lemma~\ref{lem:stablemacrodscomposition}
\begin{equation}\label{eq:temp6}
\boldsymbol{\phi}=\sum_{k=1}^J\boldsymbol{\phi}_k, \;\textrm{ and }\; \sum_{k=1}^J\|\nabla ^s\times\boldsymbol{\phi}_k\|_0^2\eqsim \|\nabla ^s\times\boldsymbol{\phi}\|_0^2.
\end{equation}
Then we apply Lemma~\ref{lem:stablemicrodscomposition} to obtain a decomposition of $\boldsymbol{\phi}_k$
such that
\begin{equation}\label{eq:temp7}
\boldsymbol{\phi}_k=\sum_{i=1}^{N_k}\boldsymbol{\phi}_{k,i}, \;\textrm{ and }\; \sum_{i=1}^{N_k}\|\nabla ^s\times\boldsymbol{\phi}_{k,i}\|_0^2\lesssim \|\nabla ^s\times\boldsymbol{\phi}_k\|_0^2
\end{equation}
with $\boldsymbol{\phi}_{k,i}\in \mathcal{S}_{k,i}$ for $i=1,2,\cdots, N_k$ and $k=1,2,\cdots, J$.
Now let $\boldsymbol{\sigma}_{k,i}=\nabla ^s\times\boldsymbol{\phi}_{k,i}\in\mathcal{K}_{k,i}$.
It is apparent that
\[
\boldsymbol{\sigma}=\sum_{k=1}^{J}\sum_{i=1}^{N_k}\boldsymbol{\sigma}_{k,i}.
\]
Moreover, it follows from \eqref{eq:temp6}-\eqref{eq:temp7}
\[
\sum_{k=1}^{J}\sum_{i=1}^{N_k}\|\boldsymbol{\sigma}_{k,i}\|_0^2=\sum_{k=1}^{J}\sum_{i=1}^{N_k}\|\nabla ^s\times\boldsymbol{\phi}_{k,i}\|_0^2\lesssim \sum_{k=1}^J\|\nabla ^s\times\boldsymbol{\phi}_k\|_0^2\lesssim \|\nabla ^s\times\boldsymbol{\phi}\|_0^2= \|\boldsymbol{\sigma}\|_0^2.
\]
The proof is ended.
\end{proof}

\section{Multigrid Methods for the HHJ mixed method}
In this section we shall develop a multigrid method using an overlapping Schwarz smoother for the HHJ mixed method and prove its uniform convergence.
We first solve a Poisson equation with a Dirichlet boundary condition to transfer the source.
Then we apply the multilevel method advised in~\cite{ChenSaddle} and the space decomposition \eqref{eq:Kdec} to obtain a V-cycle multigrid method with an overlapping Schwarz smoother for the HHJ mixed method. We analyze the V-cycle multigrid method by using the stable decomposition and the strengthened Cauchy Schwarz inequality.

\subsection{Reformulation}
We change the source to the first equation in the saddle point system~\eqref{mfem1}-\eqref{mfem2}. One possibility is as follows: let $w_h\in \mathcal{P}_h$ be the solution of
\[
\int_{\Omega}\boldsymbol{\nabla}w_h\cdot \boldsymbol{\nabla}v_h\,{\rm d}x=\int_{\Omega}fv_h\, {\rm d}x \quad \forall~v_h\in\mathcal{P}_h.
\]
This is the standard Poisson equation which can be solved efficiently by multigrid methods. Let $\boldsymbol{\sigma}_0=\left(
\begin{array}{cc}
w_h & 0 \\
0 & w_h
\end{array}
\right)$. According to the proof of Lemma~\ref{lem:hhjexactsequence}, we have
$M_n(\boldsymbol{\sigma}_0)=w_h$, $\boldsymbol{\sigma}_0\in\mathcal{V}$, $\boldsymbol{\Pi}_h\boldsymbol{\sigma}_0\in\mathcal{V}_h$ and $(\mathrm{div}\boldsymbol{\mathrm{div}})_h\boldsymbol\Pi_h \boldsymbol{\sigma}_0 = - Q_h f$, i.e.,
\[
b(\boldsymbol{\Pi}_h\boldsymbol{\sigma}_0, v_h)=-\int_{\Omega}fv_h\, {\rm d}x\quad \forall v_h \in \mathcal P_h.
\]
Now set $\boldsymbol{\sigma}_h=\widetilde{\boldsymbol{\sigma}}_h+\boldsymbol{\Pi}_h\boldsymbol{\sigma}_0$, then the HHJ mixed method~\eqref{mfem1}-\eqref{mfem2} is equivalent to: Find $(\widetilde{\boldsymbol{\sigma}}_h, u_h)\in \mathcal{V}_h\times \mathcal{P}_h$ such that
\begin{align}
&a(\widetilde{\boldsymbol{\sigma}}_h,\boldsymbol{\tau})+b(\boldsymbol{\tau}, u_h)=-a(\boldsymbol{\Pi}_h\boldsymbol{\sigma}_0,\boldsymbol{\tau}) \quad \forall\,\boldsymbol{\tau}\in\mathcal{V}_h, \label{mfemnew1} \\
&b(\widetilde{\boldsymbol{\sigma}}_h, v)\quad\quad\quad\quad\;\;\;=0 \quad\quad\quad\quad\quad\quad\; \forall\,v\in \mathcal{P}_h. \label{mfemnew2}
\end{align}
After obtaining $\boldsymbol{\sigma}_h$, due to Theorems~5.1-5.2 in~\cite{KrendlRafetsederZulehner2016},  we can acquire deflection by solving the following Poisson problem using standard multigrid methods:
Find $u_h\in \mathcal{P}_h$ such that
\[
\int_{\Omega}\boldsymbol{\nabla}u_h\cdot \boldsymbol{\nabla}v_h\,{\rm d}x=a(\boldsymbol{\sigma}_h, \boldsymbol{\Pi}_h\boldsymbol{\tau}_0) \quad \forall~v_h\in\mathcal{P}_h.
\]
with $\boldsymbol{\tau}_0=\left(
\begin{array}{cc}
v_h & 0 \\
0 & v_h
\end{array}
\right)$.

Our multigrid method is actually developed for solving \eqref{mfemnew1}-\eqref{mfemnew2}.

\subsection{A V-cycle Multigrid Method}
We shall use the multilevel methods for constrained minimization problems developed in~\cite{ChenSaddle} and adapt to the HHJ mixed method under consideration.
For simplicity, we consider the lowest order HHJ mixed method for which $\mathcal V_h$ consists of piecewise constant symmetric matrix function and normal-normal component is continuous, $S_h$ is the standard linear finite element space for vector functions, and $\mathcal P_h$ is the linear finite element space with zero boundary condition for scalar functions. For the high order HHJ mixed method, we can combine the multigrid cycles for the lowest order and an overlapping Schwarz smoother in the finest level to design efficient multigrid solvers.

Let $\boldsymbol{M}_k:\mathcal{V}_k\to\mathcal{V}_k$ be the mass operator associated with the bilinear form $a(\cdot,\cdot)$: for any $\boldsymbol{\tau}\in\mathcal{V}_k$, $\boldsymbol{M}_k\boldsymbol{\tau}\in\mathcal{V}_k$ is uniquely determined by
\[
\int_{\Omega}\boldsymbol{M}_k\boldsymbol{\tau}:\boldsymbol{\varsigma}\,{\rm d}x=a(\boldsymbol{\tau}, \boldsymbol{\varsigma})\quad \forall~\boldsymbol{\varsigma}\in\mathcal{V}_k.
\]
The mixed variational problem in the $k$-th level is:
Find $(\widetilde{\boldsymbol{\sigma}}_k, u_k)\in \mathcal{V}_k\times \mathcal{P}_k$ such that
\begin{equation}\label{hhjk}
\begin{aligned}
&a(\widetilde{\boldsymbol{\sigma}}_k,\boldsymbol{\tau})+b(\boldsymbol{\tau}, u_k)=\int_{\Omega}\boldsymbol{r}_k:\boldsymbol{\tau}\, {\rm d}x \quad \forall\,\boldsymbol{\tau}\in\mathcal{V}_k,  \\
&b(\widetilde{\boldsymbol{\sigma}}_k, v)\quad\quad\quad\quad\;\;\;=0 \quad\quad\quad\quad\quad\quad\; \forall\,v\in \mathcal{P}_k,
\end{aligned}
\end{equation}
with the residual $\boldsymbol{r}_k\in\boldsymbol{L}^2(\Omega; \mathbb{S})$.

As we mentioned before, the smoother in each level is an overlapping Schwarz method. To simplify the notation, we skip the level index $k$ and describe the local problem in each subspace $\mathcal V_i$ (of a given level $k$) below.
Define $\boldsymbol M_i: \mathcal V_i \to \mathcal V_i$ as for $\boldsymbol \sigma_i\in \mathcal V_i$, $\boldsymbol M\boldsymbol \sigma_i\in \mathcal V_i$ such that $(\boldsymbol M_i \boldsymbol \sigma_i, \boldsymbol \tau_i) = (\boldsymbol M\boldsymbol \sigma_i, \boldsymbol \tau_i)$ for all $\boldsymbol \tau_i \in \mathcal V_i$. Let $\mathcal P_i = \mathcal P\cap \mathrm{div}\boldsymbol{\mathrm{div}}_h(\mathcal V_i)$. Define $B_i: \mathcal V_i \to \mathcal P_i$
as for $\boldsymbol \sigma_i\in \mathcal V_i$, $\mathrm{div}\boldsymbol{\mathrm{div}}_h\boldsymbol \sigma_i\in \mathcal P_i$ such that $(B_i \boldsymbol \sigma_i, v_i) = (\mathrm{div}\boldsymbol{\mathrm{div}}_h\boldsymbol \sigma_i, v_i)$ for all $v_i \in \mathcal P_i$.
\begin{equation}\label{localproblem}
\begin{pmatrix}
\boldsymbol M_i & B^T_i\\
B_i & O
\end{pmatrix}
\begin{pmatrix}
\boldsymbol e_i\\
u_i
\end{pmatrix}
=
\begin{pmatrix}
f - \boldsymbol M\boldsymbol \sigma_{i-1}\\
0
\end{pmatrix}.
\end{equation}

Let $\omega_i$ be the support of $\mathcal V_i$. For the lowest order HHJ mixed method, this is the patch of the $i$-th vertex of the triangulation in the given level. The space $\mathcal V_i$ is spanned by basis functions associated to all edges connecting to the $i$-th vertex. The matrix representation of $\boldsymbol M_i$ can be extracted from the global one using the edge index in $\omega_i$. The right-hand side is the corresponding components of $f$ minus the contribution from the current approximation. Note that $\boldsymbol M\boldsymbol\sigma_{i-1}$ only need to be computed locally by including the boundary edge index of $\partial \omega_i$. The exact space $\mathcal P_i$ is somehow difficulty to identify. We shall work on the space $\mathcal K_i$ instead. An algebraic way to find $\mathcal K_i$ is as follows. We extract a sub-matrix of $B_i$ consisting of all nonzero entries associated to the edge index in $\mathcal V_i$ and compute $\ker(B_i)$ numerically. An alternative way is computing $\nabla ^s\times\boldsymbol \phi_i$ where $\boldsymbol \phi_i$ is the vector hat function associated to vertex $i$.

\smallskip
\begin{remark}\rm
Since $\ker(B_h)=\nabla ^s\times\mathcal{S}_h$ due to the exact sequence \eqref{exactsequence2}, the mixed method \eqref{mfemnew1}-\eqref{mfemnew2} can be rewritten as: Find $\boldsymbol{\phi}_h\in \widetilde{\mathcal{S}}_h$ such that
\begin{equation}\label{differentMG}
a(\nabla ^s\times\boldsymbol{\phi}_h, \nabla ^s\times\boldsymbol{\psi})=-a(\boldsymbol{\Pi}_h\boldsymbol{\sigma}_0, \nabla ^s\times\boldsymbol{\psi}) \quad \forall\,\boldsymbol{\psi}\in\widetilde{\mathcal{S}}_h
\end{equation}
with $\widetilde{\boldsymbol{\sigma}}_h=\nabla ^s\times\boldsymbol{\phi}_h$. By the theory in~\cite{LeeWuXuZikatanov2008}, this symmetric and positive semidefinite problem can be solved by multigrid methods efficiently. Solving the mixed method \eqref{mfemnew1}-\eqref{mfemnew2} is essentially equivalent to the multigrid method developed for \eqref{differentMG}. $\Box$
\end{remark}
\smallskip

We then discuss the prolongation and restriction operators. Since both finite element spaces $\mathcal{V}_k$ and $\mathcal{P}_k$ are nested,
 the prolongations $\boldsymbol{I}_{k-1}^k: \mathcal{V}_{k-1}\to \mathcal{V}_k$ and $I_{k-1}^k: \mathcal{P}_{k-1}\to \mathcal{P}_{k}$ are chosen as the natural inclusions.
Set the restriction $\boldsymbol{I}_k^{k-1}:=(\boldsymbol{I}_{k-1}^k)^T$ and $I_k^{k-1}:=(I_{k-1}^k)^T$.
With the restriction and prolongation matrix, the matrices $\boldsymbol M_k$ and $B_k$ in each level can be obtained by the standard triple product.

With previous preparation, a V-cycle multigrid method for problem~\eqref{mfemnew1}-\eqref{mfemnew2} is summarized in Algorithm~\ref{alg:multigrid} with $\boldsymbol{r}_J=-\boldsymbol{M}_J\boldsymbol{\Pi}_J\boldsymbol{\sigma}_0$.
\begin{algorithm}\label{alg:multigrid}
\TitleOfAlgo{ ${\rm MG}(k, \widetilde{\boldsymbol{\sigma}}_{k}, \boldsymbol{r}_k)$}
\If {$k=1$}{solve problem \eqref{hhjk} exactly\;}
\If {$k>1$}{
\emph{Presmoothing}\\
\For{$j=1:m_1$}{
$\widetilde{\boldsymbol{\sigma}}_{k,0}\leftarrow \widetilde{\boldsymbol{\sigma}}_{k}$\;
\For{$i=1:N_k$}{
Update $\widetilde{\boldsymbol{\sigma}}_{k,i}$ by solving local problem \eqref{localproblem}\;
}
$\widetilde{\boldsymbol{\sigma}}_{k}\leftarrow \widetilde{\boldsymbol{\sigma}}_{k,N_k}$\;
}
\emph{Coarse grid correction}\\
$\boldsymbol{r}_{k-1}\leftarrow\boldsymbol{I}_k^{k-1}(\boldsymbol{r}_{k}-\boldsymbol{M}_k\widetilde{\boldsymbol{\sigma}}_k)$\;
$e_{k-1}^{\widetilde{\boldsymbol{\sigma}}}\leftarrow{\rm MG}(k-1, \boldsymbol{0}, \boldsymbol{r}_{k-1})$\;
$\widetilde{\boldsymbol{\sigma}}_{k}\leftarrow \widetilde{\boldsymbol{\sigma}}_{k}+\boldsymbol{I}_{k-1}^ke_{k-1}^{\widetilde{\boldsymbol{\sigma}}}$\;
\emph{Postsmoothing}\\
\For{$j=1:m_2$}{
$\widetilde{\boldsymbol{\sigma}}_{k,0}\leftarrow \widetilde{\boldsymbol{\sigma}}_{k}$\;
\For{$i=N_k:-1:1$}{
Update $\widetilde{\boldsymbol{\sigma}}_{k,i}$ by solving local problem \eqref{localproblem}\;
}
$\widetilde{\boldsymbol{\sigma}}_{k}\leftarrow \widetilde{\boldsymbol{\sigma}}_{k,N_k}$\;
}
}
\medskip
\caption{A V-cycle multigrid method for problem~\eqref{mfemnew1}-\eqref{mfemnew2}.}
\end{algorithm}

The $A$-norm introduced by $a(\cdot,\cdot)$ on $\mathcal V_h$ is equivalent to the $L^2$-norm. With the stable decomposition and the strengthened Cauchy-Schwarz inequality proved in Section~\ref{section:sdscs}, applying the framework developed in \cite{ChenSaddle}, we concluded that multigrid method Algorithm~\ref{alg:multigrid} is a contraction with contraction number bounded away from one uniformly with respect to mesh size as follows.
\begin{theorem}\label{thm:vmgconvergencerate}
Let $(\widetilde{\boldsymbol{\sigma}}_h, u_h)$ be the solution of the mixed method \eqref{mfemnew1}-\eqref{mfemnew2}. Given an initial guess $\widetilde{\boldsymbol{\sigma}}_{0}\in \mathcal{V}_h$,  let $\widetilde{\boldsymbol{\sigma}}^{k}$ be the $k$th iteration in Algorithm~\ref{alg:multigrid}. Then there exists a constant $\delta \in (0,1)$ independent of the mesh level such that
$$
\|\widetilde{\boldsymbol{\sigma}}_h-\widetilde{\boldsymbol{\sigma}}^{k+1}\|_A^2\leq \delta \|\widetilde{\boldsymbol{\sigma}}_h-\widetilde{\boldsymbol{\sigma}}^{k}\|_A^2
$$
with $\|\boldsymbol{\tau}\|_A^2:=a(\boldsymbol{\tau}, \boldsymbol{\tau})$.
\end{theorem}

\subsection{Numerical Results}
To confirm the theoretical results established in the previous
sections, numerical experiments are carried out. The simulation is implemented using the MATLAB software package $i$FEM~\cite{Chen.L2008c}.
Set $r=1$.
Starting from an initial grid, several uniform refinement are applied to obtain a fine mesh. The level listed in the first column indicates how many refinements applied and the size of the saddle point system is listed in the second column. The stopping criterion  is the relative residual is less than $10^{-8}$. The iteration steps are reported in Table \ref{tab:iteration}.

We test two examples. One is a square $\Omega=(0,1)\times(0,1)$ and another is an L-shaped domain $\Omega=(-1,1)\times(-1,1)\backslash
[0,1)\times(-1,0]$. For the square domain, the Poisson ratio is $\nu = 0.3$ and the exact solution of \eqref{kirchhoffplate} is chosen as
\[
u(x, y)=(x^2 - x)^2( y^2 - y)^2.
\]
And for L-shaped domain, we simply set $f = 1$ and the Poisson ratio $\nu = 0$.
The later example is to test the multigrid method for problems without full regularity assumption. From Table~\ref{tab:iteration} we can see that the iteration steps of V-cycle multigrid method almostly remain invariant when the mesh size becomes smaller and smaller, as Theorem~\ref{thm:vmgconvergencerate} indicates.
Moreover through the comparison of different number of smoothings, we conclude that one smoothing is enough. Two smoothing steps will save only few iteration steps but with more computational cost since the cost of one V-cycle with $2$ smoothing steps is almost doubled that with $1$ smoothing step. This indeed shows the advantage of removing the assumption of requiring large enough smoothing steps.
These numerical results are all in coincide with the theoretical result Theorem~\ref{thm:vmgconvergencerate}.
At last, it is observed from Table~\ref{table:errors} that the convergence rates of $\|u-u_h\|_{0}$ and $|u-u_h|_1$ for the unit square example with $\nu = 0.3$ are $O(h^2)$ and $O(h)$ respectively, both of which are optimal.

\begin{table}[htbp]
\caption{Iteration steps of V-cycle multigrid for the saddle point system with $(m_1,m_2)$: $m_1$ pre-smoothing and $m_2$ post-smoothing steps. Stopping criterion  is the relative residual is less than $10^{-8}$. The left table is on the unit square example with $\nu = 0.3$ and the right one is the L-shaped domain example with $\nu = 0$.}
\begin{center}
\begin{tabular}{cccc || cccc}
\hline \hline
level & size & $(1,1)$ & $(2,2)$ & level & size & $(1,1)$ & $(2,2)$ \\
\hline
3 & 1,089 & 18 & 14 & 3 & 833 & 13 & 11\\
\hline
4 & 4,225 & 21 & 15 & 4 & 3,201 & 17 & 14\\
\hline
5 & 16,641 & 22 & 16 & 5 & 12,545 & 19 & 16\\
\hline
6 & 66,049 & 23 & 16 & 6 & 49,665 & 20 & 17\\
\hline \hline
\end{tabular}
\end{center}
\label{tab:iteration}
\end{table}%
\begin{table}[!h]
\tabcolsep 5pt \caption{Numerical errors for the unit square example with $\nu = 0.3$.}\label{table:errors} 
\begin{center}
\begin{tabular}{|c|c|c|c|c|}
\hline  level & $\|u-u_h\|_{0}$ & order & $|u-u_h|_1$ & order \\
\hline $2$ & 4.8576E-04 & $-$ & 3.2658E-03 & $-$ \\
\hline $3$ & 1.2846E-04 & 1.92 & 1.2925E-03 & 1.34 \\
\hline $4$ & 3.2667E-05 & 1.98 & 5.9046E-04 & 1.13 \\
\hline $5$ & 8.2042E-06 & 1.99 & 2.8757E-04 & 1.04 \\
\hline $6$ & 2.0534E-06 & 2.00 & 1.4280E-04 & 1.01 \\
\hline $7$ & 5.1351E-07 & 2.00 & 7.1278E-05 & 1.00 \\
\hline
       \end{tabular}
       \end{center}
\end{table}

\section{Conclusion}

In this paper, we have advanced and analyzed a V-cycle multigrid method with an overlapping Schwarz smoother for the HHJ mixed method. The novelties of our
V-cycle multigrid method are:
\begin{enumerate}[(1)]
\item Full regularity assumption is not necessary for our multigrid method, i.e. our approach works for both convex and non-convex domains.
\item One smoothing step is enough to guarantee the uniform convergence of our V-cycle multigrid algorithm, whereas large enough smoothing steps are usually required in the former multigrid methods for the fourth order partial differential equation.
\end{enumerate}

To obtain the uniform convergence of our V-cycle multigrid algorithm, we establish the exact sequence for the HHJ mixed method in both the continuous and discrete levels, and prove the stable decomposition and strengthened Cauchy Schwarz inequality.
Then using the framework developed in~\cite{ChenSaddle} we obtain the uniform convergence.


\end{document}